\RequirePackage{fix-cm}
\documentclass[twocolumn]{svjour3}          
\smartqed  
\usepackage{graphicx,amsmath,amssymb,amsfonts,enumerate,xcolor,psfrag,mathrsfs}
\journalname{Statistics and Computing}

\newcommand{\sfC}{{\mathsf C}}
\def\tn{\theta}   
\def\otn{{\overline{\theta}}} 
\def\tu{\tilde \theta}  
\def\tstar{\theta_\star} 
\def\pas{\gamma}

\newcommand{\un}{\mathbf{1}}
\def\tv{\mathrm{TV}} 
\def\eqsp{\;}
\def\Xset{\mathsf{X}}
\def\rmd{\mathrm{d}}
\newcommand{\R}{{\mathbb R}}
\newcommand{\N}{{\mathbb N}}
\newcommand{\hatH}{\widehat{H}}
\renewcommand{\P}{{\mathbb P}}

\def\E{{\mathbb E}}
\def\S{S}
\newcommand{\dps}{\displaystyle}
\renewcommand{\leq}{\leqslant}
\renewcommand{\geq}{\geqslant}

   {\ \\ {\bf Proof #1~: }}%
   {\hfill\mbox{\rule{2 true mm}{3 true mm}}\vskip 2 ex\noindent}

\newcounter{hypoconbis}
\newcounter{saveconbis}
\newcommand\debutA{\begin{list} {\textbf{A\arabic{hypoconbis}}}{\usecounter{hypoconbis}}\setcounter{hypoconbis}{\value{saveconbis}}}
\newcommand\finA{\end{list}\setcounter{saveconbis}{\value{hypoconbis}}}

\newtheorem{algo}{Algorithm}

\begin{document}

\title{Self-Healing Umbrella Sampling
\thanks{This work is supported by the European Research Council under
  the European Union's Seventh Framework Programme (FP/2007-2013) /
  ERC Grant Agreement number 614492 and by the French National
  Research Agency under the grant ANR-12-BS01-0019 (STAB).
  We also benefited from the scientific environment of the Laboratoire International Associ\'e between the Centre National de la Recherche Scientifique and the University of Illinois at Urbana-Champaign.}
}
\subtitle{Convergence and efficiency}

\author{Gersende Fort \and Benjamin Jourdain \and Tony Leli\`evre \and Gabriel Stoltz}
\authorrunning{G. Fort, B. Jourdain, T. Leli\`evre and G. Stoltz} 

\institute{ 
  G. Fort \at LTCI, CNRS \& Telecom Paris Tech
  46, rue Barrault \\
  75634 Paris Cedex 13 \\
  France \\
  \email{gersende.fort@telecom-paristech.fr}          
  \and 
  B. Jourdain, T. Leli\`evre and G. Stoltz \at
  Universit\'e Paris-Est, CERMICS (ENPC), INRIA \\
  F-77455 Marne-la-Vall\'ee \\
  France
  \email{\{jourdain,lelievre,stoltz\}@cermics.enpc.fr}          
}

\date{Received: date / Accepted: date}

\maketitle

\begin{abstract}
The Self-Healing Umbrella Sampling (SHUS) algorithm is an adaptive biasing
algorithm which has been proposed in~\cite{SHUS} in order to
efficiently sample a multimodal probability measure. We show that
this \linebreak method can be seen as a variant of the well-known Wang-Landau
algorithm~\cite{wang-landau-01,wang-landau-01-PRL}. Adapting
results on the convergence of the Wang-Landau
algorithm obtained in~\cite{fort:jourdain:kuhn:lelievre:stoltz:2014}, we prove the
convergence of the SHUS algorithm. We also compare the two methods in
terms of efficiency. We finally propose a modification of the SHUS algorithm in order to increase its efficiency, and exhibit some similarities of SHUS with the well-tempered metadynamics method~\cite{barducci-bussi-parrinello-08}.

\end{abstract}

\section{Introduction}
\label{intro}

The efficient sampling of a probability measure defined over a high
dimensional space is required in many
application fields, such as computational statistics or molecular
dynamics~\cite{lelievre-rousset-stoltz-book-10}. Standard algorithms consist in building a dynamics which is
ergodic with respect to the target distribution such as Langevin
dynamics~\cite{lelievre-rousset-stoltz-book-10,roberts-tweedie-96} or Metropolis-Hastings dynamics~\cite{MRRTT53,Hastings70}. Averages over trajectories
of this ergodic dynamics are then used as approximations of averages
with respect to the
target probability measure. In many cases of interest, this probability measure is
multimodal: regions of high probability are separated by regions of
low probability, and this implies that the ergodic dynamics is
metastable. This means that it takes a lot of time to leave a high probability region
(called a metastable state). The consequence of this metastable
behavior is that trajectorial averages converge very slowly to their
ergodic limits.

Many techniques have been proposed to overcome these
difficulties. Among them, importance sampling consists in modifying the target
probability using a well-chosen bias in
order to enhance the convergence to equilibrium. Averages with respect
to the original target are then recovered using a reweighting of the
biased samples. In general, it is not
easy to devise an appropriate bias. Adaptive importance sampling
methods have thus been proposed in order to automatically build a ``good'' bias
(see~\cite[Chapter~5]{lelievre-rousset-stoltz-book-10} for a review of these approaches).

Let us explain the principle of adaptive biasing techniques in a
specific setting (we refer to~\cite[Chapter
5]{lelievre-rousset-stoltz-book-10} for a more general introduction to
such methods). To fix
the ideas, let us consider a target probability measure $\pi \, \rmd \lambda$ on the state space $\Xset \subseteq
\R^D$, where $\lambda$ denotes the
Lebesgue measure on $\R^D$, and let us be given a partition $\Xset_1, \cdots, \Xset_d$ of
$\Xset$ into $d$ subsets. The subsets $\Xset_i$ are henceforth called 'strata'.  The choice of such a partition will be
discussed later (see Footnote~\ref{foot} below).
We assume that the target measure is multimodal in
the sense that the weights of the strata span several orders of magnitude. In other
words, the weights of the strata
\begin{equation}
  \label{eq:def:thetastar}
  \tstar(i) =  \int_{\Xset_i} \pi(x) \, \rmd \lambda(x), \quad  i=1,\ldots,d
\end{equation}
vary a lot and $\frac{\max_{1 \leq i \leq d}   \tstar(i)}{
  \min_{1 \leq i \leq d}   \tstar(i)}$ is large. We
suppose from now on and without loss of generality (up to removing some strata) that $\min_{1 \leq i \leq d}
\tstar(i) > 0$.  In such a situation,
it is natural to consider the following biased probability density:
\begin{equation}\label{eq:def:pithetastar}
\pi_{\tstar}(x)=  \frac{1}{d}  \sum_{i=1}^d  \frac{\pi(x)}{\tstar(i)} \un_{\Xset_i}(x) \eqsp,
\end{equation}
which is such that 
\begin{equation}
  \label{eq:uniform_strata}
  \int_{\Xset_i} \pi_{\tstar}(x) \, \rmd \lambda(x) = \frac{1}{d}
\end{equation}
for all $i \in \{1, \ldots, d\}$: under the biased
probability measure $\pi_{\tstar} \, \rmd \lambda$, each stratum has the same weight. In particular,
the ergodic dynamics which are built with
$\pi_{\tstar}$ as the target measure
are typically  less metastable than the dynamics with  target
$\pi$. The practical difficulty to implement this technique is of course that the vector
\[
\tstar= \Big\{\tstar(1), \ldots, \tstar(d) \Big\}
\]
is unknown. The principle of
adaptive biasing methods is  to learn on the fly the vector
$\tstar$ in order to eventually sample the biased probability measure
$\pi_{\tstar}$. Adaptive algorithms thus build a sequence of vectors
$(\tn_n)_{n \geq 0}$ which is expected to converge to $\tstar$.
Various updating procedures have been proposed~\cite[Chapter~5]{lelievre-rousset-stoltz-book-10}. Such
adaptive techniques are used on a daily basis by
many practitioners in particular for free energy computations in computational statistical
physics. In this context, the partition of $\Xset$ is related to the
choice of a so-called reaction coordinate\footnote{\label{foot}For a given measure
  $\pi \, \rmd \lambda$, the choice of a
  good partition or, equivalently, of a good reaction coordinate does
  of course influence the efficiency of the algorithm. This choice is a
  difficult problem that we do not consider here (see for
  example~\cite{chopin-lelievre-stoltz-12} for such discussions 
  in the context of computational statistics). Here, both $\pi \, \rmd
  \lambda$ and $\Xset_1, \cdots, \Xset_d$ are assumed to be given.}, and the weights
$\tstar$ give the free energy profile associated with this reaction coordinate.
 We focus here on a specific adaptive biasing method called the Self-Healing Umbrella Sampling (SHUS)
technique~\cite{SHUS,dickson-legoll-lelievre-stoltz-fleurat-lessard-10}. We
will show that it is a variant of the well-known Wang-Landau
method~\cite{wang-landau-01,wang-landau-01-PRL}. From a practical
viewpoint, the main interest of SHUS compared to Wang-Landau is that
the practitioner has less numerical parameters to tune (as will be
explained in Section~\ref{sec:refshuswl}).

The aim of this paper is to analyze the SHUS algorithm in terms of
convergence and efficiency. First, we adapt the results
of~\cite{fort:jourdain:kuhn:lelievre:stoltz:2014} which prove the
convergence of Wang-Landau to obtain the
convergence of SHUS (see Theorem~\ref{theo:cvgshus}). Second, we
perform numerical experiments to analyze the efficiency of SHUS, in
the spirit of~\cite{amrx} where similar numerical tests are performed for the
Wang-Landau algorithm. The efficiency analysis consists in estimating the average
exit time from a metastable state in the limit when $\frac{\max_{1 \leq i \leq d} \tstar(i)}{
  \min_{1 \leq i \leq d}   \tstar(i)}$ goes to infinity. Adaptive
techniques (such as SHUS or Wang-Landau) yield exit times which are
much smaller than for the original non-adaptive dynamics.

The main output of this work is that, both in terms of convergence (longtime behavior) and
efficiency (exit times from metastable states), SHUS is essentially equivalent to 
the Wang-Landau algorithm for a
specific choice of the numerical parameters. These numerical
parameters are not the optimal ones in terms of efficiency and we propose
in Section~\ref{sec:acc} a modified SHUS algorithm which is (in the longtime limit) equivalent
to the Wang-Landau algorithm with better sets of parameters.

\medskip

This article is organized as follows. In Section~\ref{sec:1}, we introduce the SHUS algorithm, check its asymptotic correctness and explain how it can be seen as a Wang-Landau algorithm with stochastic stepsize sequence. In Section 3, we state a convergence result for Wang-Landau algorithms with general (either deterministic or stochastic) stepsize sequences and deduce the convergence of SHUS. The proofs are based on stochastic approximation arguments and 
postponed to Section~\ref{sec:proofs}. Numerical results illustrating
the efficiency of the algorithm and comparing its performance with the
standard Wang-Landau algorithm are provided in
Section~\ref{sec:numerical}. Finally, in Section~\ref{sec:SHUS_WL}, we draw some conclusions on the
interest of SHUS compared with the Wang-Landau algorithm, 
and further compare SHUS with other adaptive techniques, such as the well-tempered metadynamics algorithm~\cite{barducci-bussi-parrinello-08} and the above-mentionned modified SHUS algorithm. We also prove the convergence of this modified SHUS algorithm (see Proposition \ref{propshusa}), and present numerical results showing that this new method is closely related to a Wang-Landau dynamics with larger stepsizes.

\section{The SHUS algorithm}
\label{sec:1}

Using the notation of the introduction, we consider a target probability measure $\pi \, \rmd \lambda$ on the state space $\Xset \subseteq
\R^D$ and a partition of $\Xset$ into $d$ strata $\Xset_1, \cdots, \Xset_d$.

We introduce a family of biased densities $\pi_\tn$, for $\tn \in \Theta$,
where $\Theta$ is the set of positive probability measures on $\{1, \cdots, d \}$:
\[
\Theta = \left\{ 
\tn = (\tn(1), \cdots, \tn(d)) \in (0,1)^d, \ \ \sum_{i=1}^d \tn(i) = 1 \right\} \eqsp.
\]
The biased
densities $(\pi_\tn)_{\tn \in \Theta}$ are obtained from
$\pi$ by a
reweighting of each stratum:
\begin{equation}
  \label{eq:def:PiTheta}
 \pi_\tn(x) =  \left( \sum_{j=1}^d \frac{\tstar(j)}{\tn(j)} \right)^{-1}  \sum_{i=1}^d \frac{\pi(x)}{\tn(i)} \un_{\Xset_i}(x)\eqsp,
\end{equation}
where $\theta_\star \in \Theta$ is defined by~\eqref{eq:def:thetastar}.
Observe that for any $\tn \in \Theta$ and $i \in \{1, \ldots, d \}$,
\begin{equation}
  \label{eq:PiThetaStratum}
  \int_{\Xset_i} \pi_\tn(x) \,\rmd \lambda(x) = \frac{\tstar(i)/\tn(i)}{\sum_{j=1}^d \tstar(j)/\tn(j)}.
\end{equation}
Equations~\eqref{eq:def:pithetastar} and~\eqref{eq:uniform_strata} are
respectively Equations~\eqref{eq:def:PiTheta}
and~\eqref{eq:PiThetaStratum} with the specific choice $\theta=\theta_\star$.

\subsection{Description of the algorithm}\label{secdesal}

As explained above, the principle of many adaptive biasing techniques, and SHUS in particular, is to build a sequence $(\theta_n)_{n
  \geq 1}$ which converges to $\theta_\star$. This allows to 
sample $\pi_{\tstar}$, which is less multimodal than $\pi$. In order to understand how the
updating rule for $\theta_n$ is built for SHUS, one may proceed as follows.

Let us first assume that we are given a Markov chain $(X_n)_{n \geq 0}$
which is ergodic with respect to the target measure $\pi \, \rmd
\lambda$ (think of a Metropolis-Hastings dynamics). Let us introduce
the sequence (for a given $\gamma > 0$)
  \begin{align}
    \label{eq:def:thetatilde1}
  \tu_{n+1}(i)& = \tu_n(i) + \pas \, \un_{\Xset_i}(X_{n+1})  \\ \nonumber
& = \begin{cases}\tu_{n}(i)+{\pas} \mbox{ if }X_{n+1}\in \Xset_i \eqsp,\\
\tu_{n}(i) \mbox{ otherwise \eqsp,}\end{cases} 
  \end{align}
which, in some sense, counts the
number of visits to each stratum.
By the ergodic property, 
it is straightforward to check that
$       \tn_n = \frac{\tu_n}{\sum_{j=1}^d \tu_n(j)}$
converges almost surely (a.s.) to $\tstar$ as $n \to \infty$. As explained in the introduction, the difficulty with this algorithm is
that the convergence of $\tn_n$ to $\tstar$ is very slow due to the
metastability of the density $\pi$ and thus of the Markov chain $(X_n)_{n \geq 0}$.

The idea is then that if  an estimate  $\otn$ of $\tstar$ is available, one
should instead consider a Markov chain $(X_n)_{n \geq 0}$ which is
ergodic with respect to $\pi_{\otn}$ and thus hopefully less metastable. To
estimate $\tstar$ with this new Markov chain, one should
modify the updating rule~\eqref{eq:def:thetatilde1} as
  \begin{equation}
    \label{eq:def:thetatilde2}
  \tu_{n+1}(i) = \tu_n(i) + \pas \, \otn(i) \un_{\Xset_i}(X_{n+1})
  \end{equation}
in order to unbias the samples $(X_n)_{n \geq 0}$ (since
$\frac{\pi(x)}{\pi_{\overline{\theta}}(x)} = \left( \sum_{j=1}^d \frac{\tstar(j)}{\otn(j)} \right)\sum_{i=1}^d \otn(i) \un_{\Xset_i}(x)$). Again, by the
ergodic property, one easily gets that $\tn_n = \frac{\tu_n}{\sum_{j=1}^d \tu_n(j)}$ converges a.s. to $\tstar$ as $n \to \infty$. Indeed, $n^{-1}\sum_{k=1}^{n}\un_{\Xset_i}(X_k)$ converges a.s. to
$\int_{\Xset_i} \pi_\otn(x) \,\rmd \lambda(x)$ (given by (\ref{eq:PiThetaStratum})) as $n\to\infty$. Since
\begin{equation}
  \label{eq:RecurrenceTildeTheta}
  \frac{\tu_n(i)}{n}=\frac{\tu_0(i)}{n}+\pas \frac{\otn(i)}{n} \sum_{k=1}^{n}\un_{\Xset_i}(X_k) \eqsp, 
\end{equation}
$(\tu_n(i)/n)_{n \geq 0}$ converges a.s. to $\frac{\pas
  \tstar(i)}{\sum_{j=1}^d\tstar(j)/\otn(j)}$, which implies in turn 
\[
\lim_{n\to+\infty} \frac{1}{n} \sum_{i=1}^d \tu_n(i)=\frac{\pas}{\sum_{j=1}^d \tstar(j)/\otn(j)} \quad \mathrm{a.s.}
\]
Hence  $       \tn_n = \frac{\tu_n}{\sum_{j=1}^d \tu_n(j)}$
converges a.s. to $\tstar$ as $n \to \infty$.

The SHUS algorithm consists
in using the current
value~$\theta_n$ as the estimate $\otn$ of $\tstar$ in the previous algorithm. Let us now precisely define the SHUS algorithm.  Let
$(P_\tn)_{\tn \in \Theta}$ be a family of transition kernels on~$\Xset$ which are ergodic with respect to $\pi_\tn$. In particular, 
\[
\forall \tn \in \Theta, \quad \pi_\tn P_\tn = \pi_\tn.
\] Let
$\pas>0$, $X_0 \in \Xset$ and
$\tu_0=(\tu_0(1),\ldots,\tu_0(d))\in(\R_+^*)^d$ be deterministic.
The SHUS algorithm consists in iterating the following steps:
\begin{algo}\label{algoshus}
 Given $(\tu_n,X_n)\in
(\R_+^*)^d\times \Xset$,
\begin{itemize}
   \item compute 
     the probability measure on $\{1,\ldots,d\}$,
     \begin{equation}
       \label{eq:def:theta}
       \tn_n = \frac{\tu_n}{\sum_{j=1}^d \tu_n(j)} \in \Theta\eqsp,
     \end{equation}
   \item draw $X_{n+1}$ according to the kernel $P_{\tn_n}(X_{n},\cdot)$,
   \item compute, for all $i \in \{1, \ldots d\}$,
  \begin{equation}
    \label{eq:def:thetatilde}
  \tu_{n+1}(i) = \tu_n(i) + \pas \,  \tn_n(i) \un_{\Xset_i}(X_{n+1})
  \eqsp .
  \end{equation}
\end{itemize}
\end{algo}
Notice that only simulation under $P_\tn$ is needed to implement SHUS. By
choosing $P_\tn$ as a Metropolis-Hastings kernel with target measure $\pi_\tn\rmd \lambda$, 
the density $\pi_\tn$ needs to be known only up to
a normalizing constant. Therefore, SHUS covers the case when the
target measure $\pi \, \rmd
\lambda$ is only known up to a multiplicative constant
which is generally the case in view of
applications to Bayesian statistics and statistical physics.  

As proved in Section~\ref{sec:CvgResults}, $(\tn_n)_{n \geq 0}$
converges almost surely to $\tstar$ when $n \to \infty$. According to the update mechanism~(\ref{eq:def:thetatilde}),
$\tu_{n+1}(i) - \tu_n(i)$ is non negative, and it is 
positive if and only if the current draw $X_{n+1}$ lies in stratum $i$.
In addition, this variation is proportional to the current approximation
$\tn_n(i)$ of $\tstar(i)$ with a factor $\pas$ chosen by the user (prior to the
run of the algorithm; the choice of~$\pas$ is numerically investigated 
in Section~\ref{sec:numerical}). 

The principle of this algorithm is
thus to penalize the already visited strata, in order to favor
transitions towards unexplored strata. The penalization strength is
proportional to the current bias of the strata. The prefactor
$\theta_n(i)$ in~\eqref{eq:def:thetatilde} can be understood as a way
to unbias the samples $(X_n)_{n \geq 0}$ in order to recover samples
distributed according to the target measure $\pi \, \rmd \lambda$, see also formula~\eqref{theo:ergoLLN}
  below.

\subsection{Asymptotic correctness}
\label{sec:asymptotic_consistency}

Let us consider
the SHUS algorithm~\ref{algoshus} and let us assume that $(\tn_n)_{n \geq 0}$ converges to some value, say
$\otn \in \Theta$. It is then expected that $n^{-1} \sum_{k=1}^{n}\un_{\Xset_i}(X_k)$
converges to $\int_{\Xset_i} \pi_\otn(x) \, \rmd\lambda(x)$ and that the
updating rule~\eqref{eq:def:thetatilde} leads to the same asymptotic behavior
as~\eqref{eq:def:thetatilde2} (where $\theta_n(i)$ in~\eqref{eq:def:thetatilde} has been replaced by its
limit $\otn(i)$). Thus, it is expected that $\lim_n
{n}^{-1}\tu_n=\frac{\pas \tstar}{\sum_{j=1}^d\tstar(j)/\otn(j)}$ a.s. by
(\ref{eq:RecurrenceTildeTheta}) and thus $\lim_n \tn_n=\tstar$ a.s.. Therefore, the only possible limit of the sequence $(\tn_n)_{n \geq 0}$
is $\tstar$.

This heuristic argument is of course not a proof of convergence,
but it explains why one
can expect the SHUS algorithm to behave like a Metropolis-Hastings
algorithm with target measure $\pi_{\tstar}$. The rigorous result
for the convergence is given in Section~\ref{sec:CvgResults}, and the efficiency of
SHUS is discussed in Section~\ref{sec:numerical}.

\subsection{Reformulation as a Wang-Landau algorithm with a stochastic
  stepsize sequence}\label{sec:refshuswl}

One key observation of this work is that SHUS can be seen as a Wang-Landau algorithm with nonlinear update of the
weights and with a specific stepsize sequence $(\gamma_n)_{n \geq
  1}$, see~\cite{wang-landau-01,wang-landau-01-PRL,fort:jourdain:kuhn:lelievre:stoltz:2014,amrx}. The
Wang-Landau algorithm with nonlinear update of the
weights consists in replacing the updating formula~\eqref{eq:def:thetatilde} by:
\begin{align}
    \label{eq:def:thetatilde_WL}
  \tu^{\rm WL}_{n+1}(i) = \tu^{\rm WL}_n(i) \Big(1 + \gamma^{\rm WL}_{n+1} \un_{\Xset_i}(X_{n+1})\Big),
  \end{align}
where the deterministic stepsize sequence $(\gamma^{\rm WL}_n)_{n \geq 1}$ has to be
{\em chosen} by the practitioner beforehand. The choice of this sequence is not
easy: it should converge to zero when $n$ goes to infinity (vanishing
adaption) in order to ensure the convergence of the sequence
$(\tu^{\rm WL}_{n})_{n \geq 0}$, but not too fast otherwise the convergence
of $\tn^{\rm WL}_n=   \frac{\tu^{\rm WL}_n}{\sum_{i=1}^d \tu^{\rm WL}_n(i)}$ to $\tn_\star$ is
not ensured. 

\begin{remark}
In fact, the updating rule of the original Wang-Landau algorithm is more complicated 
than~\eqref{eq:def:thetatilde_WL} since the stepsizes are changed at random 
stopping times related to a quasi-uniform population of the strata and
not at every iteration, see~\cite{JR11} for a mathematical
analysis of the well-posedness of the algorithm.
\end{remark}  

Going back to SHUS, by setting
\begin{equation}
  \label{eq:def:stepsize}
  \gamma_{n+1} =\frac{{\pas}}{\sum_{j=1}^d \tu_n(j)} \eqsp, \qquad \text{ for $n\in\N$} \eqsp,
\end{equation}
it is easy to check that \eqref{eq:def:thetatilde} is equivalent to
\begin{equation}
\label{eq:SHUS_WL}
  \tu_{n+1}(i) = \tu_n(i) \Big(1 + \pas_{n+1} \un_{\Xset_i}(X_{n+1})\Big),  
\end{equation}
which explains why SHUS can be seen as a Wang-Landau algorithm with nonlinear update of the
weights (see~\eqref{eq:def:thetatilde_WL}) but with a stepsize sequence $(\gamma_n)_{n \geq 1}$ 
which is not chosen by the practitioner: it is adaptively built by the algorithm.


\section{Convergence results}
\label{sec:CvgResults}
The goal of this section is to establish convergence results on the sequence
$(\tn_n)_{n \geq 0}$ given by (\ref{eq:def:theta}) to the weight vector $\tstar$ and on
the distribution of the samples $(X_n)_{n \geq 0}$. To do so, we will extend the results of~\cite{fort:jourdain:kuhn:lelievre:stoltz:2014} and more generally prove convergence of the Wang-Landau algorithm with general (either deterministic or stochastic) stepsize sequence $(\gamma_n)_{n \geq 1}$ and the nonlinear update of the weights.
We need the following assumptions on the target density $\pi$ and the kernels
 $P_\theta$:
\debutA
\item \label{hyp:targetpi} The density $\pi$ of the target
  distribution is such that $0 < \inf_{\Xset}\pi \leq \sup_\Xset \pi <
  \infty$ and the strata $(\Xset_i)_{i \in \{1, \ldots ,d\}}$ satisfy $\min_{1 \leq i \leq d }
  \lambda(\Xset_i)> 0$.  \finA 

\noindent
Notice that this assumption implies that
  $\tstar$ given by~(\ref{eq:def:thetastar}) is such that $\min_{1 \leq i \leq
    d }\tstar(i)>0$ (hence $\tstar \in \Theta$).

\debutA
\item \label{hyp:kernel} For any $\tn \in \Theta$, $P_\tn$ is a
  Metropolis-Hastings transition kernel with proposal kernel $q(x,y)\,\rmd
  \lambda(y)$ where $q(x,y)$ is symmetric and satisfies $\inf_{\Xset^2} q > 0$,
  and with invariant distribution $\pi_\tn \, \rmd \lambda$, where $\pi_\tn$ is
  given by (\ref{eq:def:PiTheta}).  \finA 
\subsection{Convergence of Wang-Landau algorithms with
  a general stepsize sequence}
\label{sec:cvgWL}

In~\cite{fort:jourdain:kuhn:lelievre:stoltz:2014}, we consider the Wang-Landau
algorithm with a linear update of the weights. The linear update version of
Wang-Landau consists in changing the updating rule~\eqref{eq:def:thetatilde_WL}
to a linearized version (in the limit $\gamma^{\rm WL}_{n+1}\to 0$) on the normalized weights:
\begin{align}
\tn^{\rm WL}_{n+1}(i)&=\tn^{\rm WL}_n(i) \label{eq:linear_update}\\
&\quad+\gamma^{\rm WL}_{n+1}\tn^{\rm WL}_n(i) \left( \un_{\Xset_i}(X_{n+1}) -
  \tn^{\rm WL}_n(I(X_{n+1}))\right). \nonumber
\end{align}
 We prove in~\cite{fort:jourdain:kuhn:lelievre:stoltz:2014} that, when the target density $\pi$ and the kernels
 $P_\theta$ satisfy A\ref{hyp:targetpi} and A\ref{hyp:kernel},  the Wang-Landau algorithm with this linear update of
 the weights converges under the following
condition on $(\gamma^{\rm WL}_n)_{n \geq 1}$: the sequence $(\gamma^{\rm WL}_n)_{n \geq 1}$ is deterministic, ultimately non-increasing,
$$\sum_{n \geq 1} \gamma^{\rm WL}_n =
+\infty\mbox{ and }\sum_{n\geq 1}(\gamma^{\rm WL}_n)^2< +\infty.$$ 

To the best of knowledge, neither the convergence of the Wang-Landau algorithm with the nonlinear update of the weights~\eqref{eq:def:thetatilde_WL} nor the case of a random sequence of stepsizes are addressed in the literature. It is a particular case of the following Wang-Landau algorithm with general stepsize sequence which also generalizes SHUS: starting from random variables $\tu_0=(\tu_0(1),\ldots,\tu_0(d))\in(\R_+^*)^d$ and $X_0\in\Xset$, iterate the following steps:
\begin{algo}\label{algogen}
 Given $(\tu_n,X_n)\in
(\R_+^*)^d\times \Xset$,
\begin{itemize}
   \item compute 
     the probability measure on $\{1,\ldots,d\}$,
     \begin{equation}
       \label{eq:def:thetagen}
       \tn_n = \frac{\tu_n}{\sum_{j=1}^d \tu_n(j)} \in \Theta\eqsp,
     \end{equation}
   \item draw $X_{n+1}$ according to the kernel $P_{\tn_n}(X_{n},\cdot)$,
   \item compute, for all $i \in \{1, \ldots d\}$,
  \begin{equation}
    \label{eq:def:thetatildegen}
  \tu_{n+1}(i) = \tu_n(i)(1+\gamma_{n+1}\un_{\Xset_i}(X_{n+1}))
  \eqsp ,
  \end{equation}
\end{itemize}
where the positive stepsize sequence $(\gamma_{n})_{n\geq 1}$ is supposed to be predictable with respect to the filtration ${\cal F}_n=\sigma(\tu_0,X_0,X_1,\ldots,X_n)$ (i.e. $\gamma_{n}$ is ${\cal F}_{n-1}$-measurable).
\end{algo}
Algorithm~\ref{algogen} is a meta-algorithm: to obtain a practical algorithm, one has to specify the way the stepsize sequence $(\gamma_{n})_{n\geq 1}$ is generated.
\begin{definition}\label{defconvalgo}
  An algorithm of the type described in Algorithm~\ref{algogen} is said to converge if it satisfies the following properties:
  \begin{enumerate}[(i)]
  \item \label{theo:cvg:theta:item1} $\dps \P\left( \lim_{n \to +\infty} \tn_n
      = \tstar\right)=1$.
\item \label{theo:cvg:theta:item2} For any bounded measurable function $f$ on $\Xset$,
\begin{align}
    \lim_{n\to\infty} \E\left[f(X_n)\right] & = \int_\Xset f(x) \, \pi_{\tn_\star}(x) \, 
 \rmd \lambda( x)
    \eqsp, \notag \\
    \lim_{n\to\infty}  \frac{1}{n} \sum_{k=1}^n f(X_k) & = \int_\Xset
    f(x) \, \pi_{\tn_\star}(x) \, \rmd\lambda( x)  \quad \mathrm{a.s.}\eqsp \nonumber
  \end{align}
\item \label{theo:cvg:theta:item3} For any bounded measurable function $f$ on $\Xset$,
\begin{align}
  & \lim_{n\to\infty} \E\left[d \sum_{j=1}^d  \theta_{n-1}(j)\,f(X_n)
    \un_{\Xset_j}(X_n)\right]  \nonumber \\
  & \qquad \qquad = \int_\Xset f(x) \, \pi(x)
  \, \rmd \lambda(x)
  \eqsp,  \nonumber  
\\
  &  \lim_{n\to\infty}\frac{d}{n} \sum_{k=1}^n \sum_{j=1}^d \theta_{k-1}(j) \, \un_{\Xset_j}(X_k)
  f(X_k)  \nonumber  \\
  & \qquad \qquad = \int_\Xset f(x) \, \pi(x)\, \rmd\lambda( x) \quad \mathrm{a.s.} 
  \label{theo:ergoLLN}
  \end{align}
  \end{enumerate}
\end{definition}
In addition to the almost sure convergence of $(\tn_n)_{n \geq 0}$ to $\tstar$, this definition encompasses the ergodic convergence
of the random sequence $(X_n)_{n \geq 0}$ to $\pi_{\tn_\star} \, \rmd \lambda$ and the convergence of an importance sampling-type Monte Carlo average to the
probability measure~$\pi \, \rmd \lambda$.
Our main result is the following Theorem.
\begin{theorem}
  \label{theo:cvggal}
  Assume A\ref{hyp:targetpi}, A\ref{hyp:kernel} and that there exists a deterministic sequence $(\overline{\gamma}_n)_n$ converging to $0$ such that
\begin{equation}\label{eq:stepsize:assumptions}
\P\left(\forall n\geq 1,\;\gamma_n\leq \overline{\gamma}_n,\;\sum_n \gamma_n = +\infty \eqsp,\;\sum_n \gamma_n^2 < \infty\right)=1,
\end{equation}and the stepsize sequence $(\gamma_n)_{n \geq 1}$ is a.s. non-increasing. Then Algorithm~\ref{algogen} converges in the sense of Definition~\ref{defconvalgo}. 
\end{theorem}
One immediately deduces convergence of the Wang-Landau algorithm with determistic stepsize sequence and nonlinear update~\eqref{eq:def:thetatilde_WL} under the same assumptions as for the linear update~\eqref{eq:linear_update}, which generalizes the result of~\cite{fort:jourdain:kuhn:lelievre:stoltz:2014}.
\begin{corollary}
   Assume A\ref{hyp:targetpi}, A\ref{hyp:kernel} and that the deterministic sequence $(\gamma^{\rm WL}_n)_{n \geq 1}$ is non-increasing and such that $\sum_{n\geq 1}\gamma^{\rm WL}_n=+\infty$ and $\sum_{n\geq 1}(\gamma^{\rm WL}_n)^2<\infty$. Then the Wang-Landau algorithm with nonlinear update~\eqref{eq:def:thetatilde_WL} converges in the sense of Definition \ref{defconvalgo}. 
\end{corollary}

Before outlining the proof of Theorem \ref{theo:cvggal}, let us discuss its application to the SHUS algorithm.

\subsection{Convergence of SHUS}
The stepsize sequence $\left(\gamma_n\right)_{n\geq 1}=\left(\frac{1}{\sum_{i=1}^d\tu_{n-1}(i)}\right)_{n\geq 1}$ obtained in the reformulation of the SHUS algorithm~\ref{algoshus} as a Wang-Landau algorithm is clearly decreasing since, the sequence $(\sum_{i=1}^d\tu_n(i))_{n\geq 0}$ is increasing. To apply Theorem \ref{theo:cvggal}, we also need to check \eqref{eq:stepsize:assumptions}. This is the purpose of the following proposition which is proved in Section~\ref{sec:proofs}.
\begin{proposition}\label{lem:bounds:BigThetaTilde} 
With probability one, the stepsize sequence $(\gamma_n)_{n \geq 1}$ in the SHUS algorithm~\ref{algoshus} is decreasing and for any $n \in \N$,
\[
\forall n\in\N,\;\frac{\gamma_1}{1+n \gamma_1} \leq \gamma_{n+1} \leq \frac{\gamma_1}{
  \sqrt{ 1+2n \gamma_1\min_{1\leq i\leq d}\tn_0(i)}} \eqsp.
\]
Moreover, under A\ref{hyp:targetpi} and A\ref{hyp:kernel}, there exists a random variable $\sfC$ such that
\[
\P\left(\sfC >0 \ \mathrm{and} \ \sup_{n\in\N} n^{\frac{1+\sfC}{2}} \gamma_{n+1}
< \infty \right)=1.
\]
\end{proposition}
Since $\gamma_1=\frac{\gamma}{\sum_{i=1}^d \tu_0(i)}$ where $\tu_0$ is deterministic, combining Theorem \ref{theo:cvggal} and Proposition \ref{lem:bounds:BigThetaTilde}, we obtain the following convergence result for the SHUS algorithm~\ref{algoshus}.
\begin{theorem}\label{theo:cvgshus}
 Under A\ref{hyp:targetpi} and A\ref{hyp:kernel}, the SHUS algorithm~\ref{algoshus} converges in the sense of Definition \ref{defconvalgo}.
\end{theorem}
\begin{remark}
Since the sequence $((X_n,\tu_n))_{n\geq 0}$ generated by the SHUS algorithm is a Markov chain, one easily deduces that the SHUS algorithm started from a random initial condition $(X_0,\tu_0)\in\Xset\times({\mathbb R}_+^*)^d$ also converges in the sense of Definition~\ref{defconvalgo}. 
\end{remark}

The convergence result from Theorem~\ref{theo:cvgshus} allows us to characterize the asymptotic behavior of the stepsizes. Indeed, since for $i\in\{1,\cdots,d\}$ and $n\geq 1$,  
\[
\frac{\tu_n(i)}{n}=\frac{\tu_0(i)}{n}+\frac{\pas}{n}\sum_{k=1}^n\theta_{k-1}(i)\un_{\Xset_i}(X_k) \eqsp,
\]
the property~\eqref{theo:ergoLLN} with $f(x)=\un_{\Xset_i}(x)$ implies that the SHUS algorithm generates sequences $(\tu_n)_{n \geq 0}$ and $(\gamma_n)_{n \geq 1}$ which satisfy
\begin{corollary}\label{cor:asumptpas}
  Under A\ref{hyp:targetpi} and A\ref{hyp:kernel},
  \[
  \lim_{n \to \infty} \frac{\tu_n}{n} = \frac{\gamma \tstar}{d}
  \quad \mathrm{a.s.}  \quad \mathrm{and} \quad \lim_{n \to \infty} n\gamma_n =d \quad \mathrm{a.s.}
  \]
\end{corollary}

Notice that this Corollary implies that the stepsize sequence
$(\gamma_n)_{ n \geq 1}$
scales like $d/n$ in the large $n$ limit. In Section~\ref{sec:scalings}, we will
therefore compare SHUS with the Wang-Landau algorithm implemented with
a stepsize sequence $\gamma^{\rm WL}_n=\frac{\gamma_\star}{n}$, for
some positive parameter $\gamma_\star$.

\subsection{Strategy of the proof of Theorem \ref{theo:cvggal}}
The proof of Theorem~\ref{theo:cvggal} is given in
Section~\ref{sec:proofs}. It relies on a rewriting of the updating mechanism of the sequence
$(\tn_n)_{n \geq 0}$ as a Stochastic Approximation (SA) algorithm for which
convergence results have been proven (see for
example~\cite{andrieu:moulines:priouret:2005}). Notice that
the updating formula \eqref{eq:def:thetagen}--\eqref{eq:def:thetatildegen} is equivalent to: for all $i \in \{1, \cdots, d \}$ and $n \in \N$
\begin{equation}
  \label{eq:theta:update:multiplicatif}
  \tn_{n+1}(i) =  \tn_n(i) \frac{1 + \gamma_{n+1}
  \un_{\Xset_i}(X_{n+1})}{1 + \gamma_{n+1}\tn_n(I(X_{n+1}))} \eqsp,
\end{equation}
where for all
$x \in \Xset$,
\[
I(x)=\sum_{j=1}^d j \un_{\Xset_j}(x)
\]
denotes the index of the stratum where $x$ lies.
Upon noting that $(1+a)/(1+b) = 1 + a -b + b (b-a)/(1+b)$,
(\ref{eq:theta:update:multiplicatif}) is equivalent to
\begin{equation}
  \label{eq:theta:approxsto}
  \tn_{n+1}(i) = \tn_n(i) + \gamma_{n+1} \, H_i(X_{n+1}, \tn_n) +
\gamma_{n+1} \,  \Lambda_{n+1}(i) \eqsp, 
\end{equation}
where $H: \Xset \times \Theta \to \R^d$ is defined by
\begin{equation}
  \label{eq:def:Hn}
  H_i(x,\tn) = \tn(i) \Big( \un_{\Xset_i}(x) - \tn(I(x)) \Big),
\end{equation}
and 
\begin{equation}
\label{eq:def:Lamndan}
\begin{aligned}
  \Lambda_{n+1}(i)  &= \gamma_{n+1} \ \tn_n(i) \ \tn_n(I(X_{n+1}))  \\
  & \quad \times\frac{\tn_n(I(X_{n+1})) - \un_{\Xset_i}(X_{n+1})}{1 + \gamma_{n+1} \, \tn_n(I(X_{n+1}))}\eqsp.
\end{aligned}
\end{equation}
Notice that the last term $\gamma_{n+1} \,  \Lambda_{n+1}$
in~\eqref{eq:theta:approxsto} is of the order of $\gamma_{n+1}^2$. The
recurrence relation~\eqref{eq:theta:approxsto} is thus in a standard
form to apply convergence results for SA algorithms 
(see {\it e.g.}~\cite{andrieu:moulines:priouret:2005}).

There are however three specific points in Algorithm~\ref{algogen} which make the study of this SA
recursion quite technical. A first difficulty raises from the fact that $(X_n)_{n \geq 0}$ alone is not a
Markov chain: given the past up to time~$n$, $X_{n+1}$ is generated according
to a Markov transition kernel computed at the current position $X_n$ but
controlled by the current value $\tn_n$, which depends on the whole trajectory $(\tu_0,X_0,\ldots,X_n)$. 
A second one comes from the randomness
of the stepsizes $(\gamma_n)_{n \geq 1}$. Finally, it is not clear
whether the sequence $(\tn_n)_{n \geq 0}$ remains a.s. in a compact subset of
the open set $\Theta$ so that a preliminary step when proving the convergence
of the SA recursion is to establish its recurrence, namely the fact
that the sequence $(\theta_n)_{n \geq 0}$ returns to a compact set of $\Theta$
infinitely often.

\subsection{A few crucial intermediate results}

Let us highlight a few results which are crucial to tackle these
difficulties and establish Theorem~\ref{theo:cvggal}.  The fundamental
result to address the dynamics of $(X_n)_{n \geq 0}$ controlled by  $(\theta_n)_{n \geq 0}$ is the following proposition established in \cite[Proposition 3.1.]{fort:jourdain:kuhn:lelievre:stoltz:2014} 
\begin{proposition}\label{prop:ergounif}
  Assume A\ref{hyp:targetpi} and A\ref{hyp:kernel}. There exists $\rho \in
  (0,1)$ such that
\[
\sup_{x\in\Xset} \sup_{\tn \in \Theta} \| P_\tn^n(x,\cdot) - \pi_\tn \, \rmd \lambda \|_{\tv} \leq 2 (1-\rho)^n \eqsp,
\]
where for a signed measure $\mu$, the total variation norm is defined as
\[
\| \mu \|_{\tv} = \sup_{\{f: \sup_\Xset |f| \leq 1\}} \left| \mu(f) \right|
\eqsp.
\]
\end{proposition}

In the present case, the recurrence
property  of the SA algorithm means the existence of a positive threshold such that infinitely often in $n$, the minimal weight
\begin{equation}
   \label{eq:StratumLowestWeight}
  \underline{\tn}_n = \min_{1\leq i\leq d} \tn_n(i) 
\end{equation} is larger than the threshold. Let \begin{equation}
  \label{eq:DefIn}
  I_n = \min \Big\{i \, : \, \tn_n(i)=\underline{\tn}_n \Big\}, 
\end{equation}
be the smallest index of stratum with smallest weight according to $\tn_n$ and  $(T_k)_{k\in\N}$ be the times of return to the stratum of smallest weight: $T_0=0$
and, for $k\geq 1$,
\begin{equation}
\label{eq:Tk}
T_k=\inf\Big\{n>T_{k-1}\,:\,X_n\in\Xset_{I_n}\Big\} ,
\end{equation}
with the convention $\inf\emptyset=+\infty$. We prove in
Section~\ref{sec:proofs} the following recurrence property.

\begin{proposition}
  \label{prop:stabilite}
  Assume A\ref{hyp:targetpi}, A\ref{hyp:kernel} and the existence of a deterministic sequence $(\overline{\gamma}_n)_n$ converging to $0$ such that $\P(\forall n\geq 1,\;\gamma_n\leq \overline{\gamma}_n)=1$.  Then Algorithm~\ref{algogen} is such that
  \begin{equation}
    \label{proprec}
    \P\left(\forall k\in\N,\;T_k<+\infty\ \mathrm{and} \ \limsup_{k\to\infty} \underline{\theta}_{T_{k}-1} >0\right)=1,
  \end{equation}
  and
\begin{equation}
  \label{lgtk}
  \P\Big(\exists C_T<+\infty,\;\forall k\in\N,\;T_k\leq C_T k\Big)=1.
\end{equation}
\end{proposition}

Using the recurrence property of
Proposition~\ref{prop:stabilite}, we are then able to prove
that Algorithm~\ref{algogen} converges in the sense of Definition~\ref{defconvalgo}-$(i)$ by using general convergence results for SA algorithms
given
in~\cite{andrieu:moulines:priouret:2005}. The properties in Definition~\ref{defconvalgo}-$(ii)$-$(iii)$
then follow from convergence results for adaptive Markov Chain Monte
Carlo algorithms given in~\cite{fort:moulines:priouret:2011}. See Section~\ref{sec:proofs} for the details.


\section{Numerical investigation of the efficiency}
\label{sec:numerical}  

We present in this section some numerical results illustrating
the efficiency of SHUS in terms of exit times from a metastable state. We also 
compare the performances of SHUS and of the Wang-Landau
algorithm on this specific example.

\subsection{Presentation of the model and of the dynamics}

We consider the system based on the two-dimensional potential suggested
in~\cite{PSLS03}, see also~\cite{amrx} for similar experiments on the
Wang-Landau algorithm. The state space is $\Xset = [-R,R] \times \mathbb{R}$ (with
$R=1.2$), and we denote by $x=(x_1,x_2)$ a generic element of~$\Xset$. The
reference measure $\lambda$ is the Lebesgue measure $\rmd x_1\, \rmd x_2$. The density of the target
measure reads 
\[
\pi(x) = Z^{-1} \un_{[-R,R]}(x_1) \, \mathrm{e}^{-\beta U(x_1,x_2)},
\]
for some positive inverse temperature $\beta$, with
\begin{equation}\label{eq:pot_U}
\text{
\scriptsize
$
\begin{aligned} U(x_1,x_2)
& = 3 \exp\left(-x_1^2 - \left(x_2-\frac13\right)^2\right) - 3 \exp\left(-x_1^2 - \left(x_2-\frac53\right)^2\right) \\
& \quad - 5 \exp\left(-(x_1-1)^2 - x_2^2\right) - 5 \exp\left(-(x_1+1)^2 - x_2^2\right)\\
& \quad + 0.2 x_1^4 + 0.2 \left(x_2-\frac13\right)^4, 
 \end{aligned}
$}
\end{equation}
and the normalization constant 
\[
Z = \int_{\Xset} \mathrm{e}^{-\beta U(x_1,x_2)} \, \rmd x_1 \, \rmd x_2 .
\]
We introduce $d$ strata $\Xset_\ell = (a_\ell,a_{\ell+1}) \times
\mathbb{R}$, with $a_\ell = -R + 2(\ell-1) R/d$ and $\ell=1, \ldots,
d$. Thus, the element $x=(x_1,x_2) \in \Xset$ lies in the stratum $I(x_1)=\left\lfloor
\frac{x_1+R}{2R} d \right\rfloor + 1$.

A plot of the level sets of the potential~$U$ is presented in
Figure~\ref{fig:contour}. The global minima of the
potential~$U$ are located at the points $x_- \simeq (-1.05,-0.04)$ and $x_+ \simeq
(1.05,-0.04)$ (notice that the potential is symmetric with respect to the $y$-axis).

The Metropolis-Hastings kernels are constructed using isotropic proposal moves
distributed in each direction according to independent Gaussian random variables with variance $\sigma^2$. 
The reference Metropolis-Hastings dynamics
$P_{(\frac{1}{d},\cdots,\frac{1}{d})}$ is ergodic and reversible with respect
to the measure with density~$\pi$. This dynamics  is metastable: for
local moves ($\sigma$ of the order of a fraction of $\|x_+-x_-\|$), it takes a
lot of time to go from the left to the right, or equivalently from the right to the left. More
precisely, there are two main metastable states: one located around $x_-$, 
and another one around $x_+$.  These two states are separated
by a region of low probability. The metastability of the dynamics increases
with $\beta$ ({\it i.e.} as the temperature decreases). The larger $\beta$ is,
the larger is the ratio between the weight under $\pi$ of the strata located
near the main metastable states and the weight under $\pi$ of the transition
region around $x_1 = 0$, and the more difficult it is to leave the left
metastable state to enter the one on the right (and conversely). We
refer for example to~\cite[Fig. 3-1]{amrx} for a numerical
quantification of this statement.

As already pointed out in~\cite{amrx}, adaptive algorithms such as the
Wang-Landau dynamics are less metastable than the original
Metropolis-Hastings dynamics, in the sense that the typical time to
leave a metastable state is much smaller thanks to the adaption mechanism.
In this section, we compare
the adaptive Markov chain $(X_n)_{n \geq 0}$ corresponding to the SHUS algorithm with the
one generated by the Wang-Landau dynamics $(X^{\rm WL}_n)_{n \geq 0}$ with a nonlinear update of the
weights (see~\eqref{eq:def:thetatilde_WL}) with stepsizes $\gamma^{\rm WL}_n=\gamma_\star/n$ for
some constant $\gamma_\star>0$. This choice for $\gamma^{\rm WL}_n $
is motivated by the asymptotic behavior
of $(\gamma_n)_{n \geq 1}$ in the limit $n \to \infty$, see Corollary~\ref{cor:asumptpas}. The proposal kernel used in the
Metropolis algorithm is the same for SHUS and Wang-Landau. Therefore,
the two algorithms only differ by the update rules of the weight sequence $(\tn_n)_{n \geq 0}$.  The initial
weight vector $\tu_0$ is $(1/d,\dots,1/d)$ and the initial condition
is $X_0 = (-1,0)$ for both dynamics.

\begin{figure}
  \includegraphics[width=0.53\textwidth]{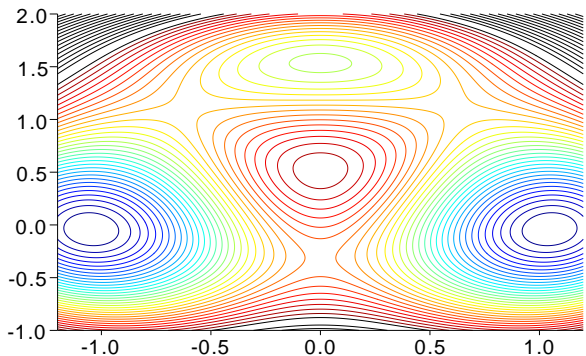}
  \caption{\label{fig:contour}  Level sets of the potential~$U$ defined
  in~\eqref{eq:pot_U}. The minima are located at the positions $x_\pm
  \simeq (\pm 1.05,-0.04)$, and there are three saddle-points, at the positions $x^{{\rm sd},1}_\pm \simeq (\pm 0.6,1.15)$ and $x^{{\rm sd},2} \simeq (0,-0.3)$. The energy differences of these saddle points with respect to the minimal potential energy are respectively~$\Delta U^1 \simeq 2.2$ and~$\Delta U^2 \simeq 2.7$. }
\end{figure}

\subsection{Study of a typical realization}

Let us first consider a typical realization of the SHUS algorithm,
in the case when $\sigma$ is equal to the width of a stratum~$2R/d$, with
$d=48$ (so that $\sigma=0.05$), $\gamma=1$ and an inverse temperature~$\beta =
10$. The values of the first component of the chain as a function of the
iterations index $n \mapsto X_{n,1}$ are reported in~Figure~\ref{fig:traj}. The
trajectory is qualitatively very similar to the trajectories obtained with the
Wang-Landau algorithm (see for example~\cite[Figure~5]{amrx}). In
particular, the exit time out of the second visited well is
much larger than the exit time out of the first one. Such a behavior is proved
for the Wang-Landau algorithm with a deterministic stepsize $\gamma_n =
\gamma_\star/n$ in~\cite[Section~6]{amrx}.  After the exit out of the second
well, the convergence of the sequence $(\theta_n)_{n \geq 0}$ is already almost achieved, and
the dynamics freely moves between the two wells.

When the initial exit out of the first well occurs, the biases within this well are already very well converged, see Figure~\ref{fig:bias}. 

In the longtime limit, according to Figure~\ref{fig:eff_time_step} (top), one has
$n\gamma_n \to d = 48$, as predicted by Corollary \ref{cor:asumptpas}. In the
initial phase (before the exit out of the left well), $n\gamma_n$ stabilizes
around a value corresponding to the number of strata explored within
the well, see Figure~\ref{fig:eff_time_step} (bottom) (here,
the exploration is well performed for values of~$x_1$ between -1.2 and~-0.45,
which corresponds to~15 strata). In this phase, only the restriction of the
target density $\pi$ to these strata is seen by the algorithm and this
convergence can be seen as a local version of Corollary \ref{cor:asumptpas}.

\begin{figure}
  \includegraphics[width=0.5\textwidth]{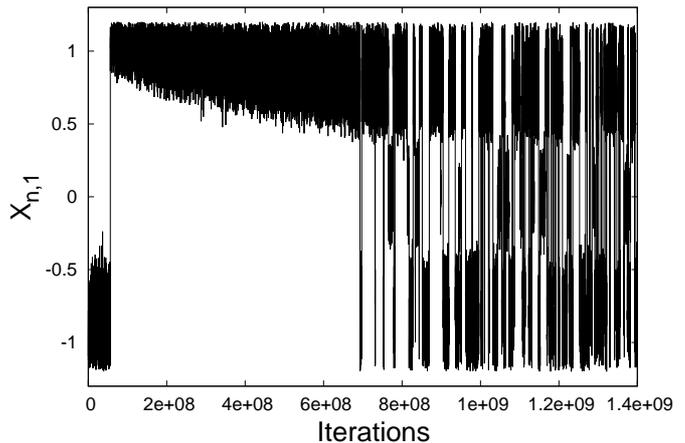}
  \caption{\label{fig:traj}  Top: Typical trajectory~$X_{n,1}$ for the parameters $d=48$, $\sigma = 2R/d = 0.05$, $\gamma = 1$ and $\beta = 10$.}
\end{figure}

\begin{figure}
  \includegraphics[width=0.5\textwidth]{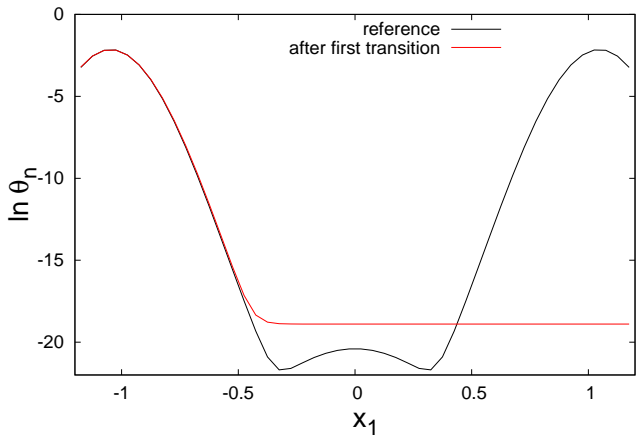}
  \caption{\label{fig:bias}  Plot of the bias $\ln(\theta_n(I(x_1)))$
    at the iteration index~$n$ when the system leaves the left well
    for the first time ($d=48$, $\sigma = 2R/d = 0.05$, $\gamma = 1$ and $\beta = 10$). The reference values $\ln(\theta_\star(I(x_1)))$ are computed with a numerical quadrature of the integrals~\eqref{eq:def:thetastar}.}
\end{figure}

\begin{figure}
  \includegraphics[width=0.5\textwidth]{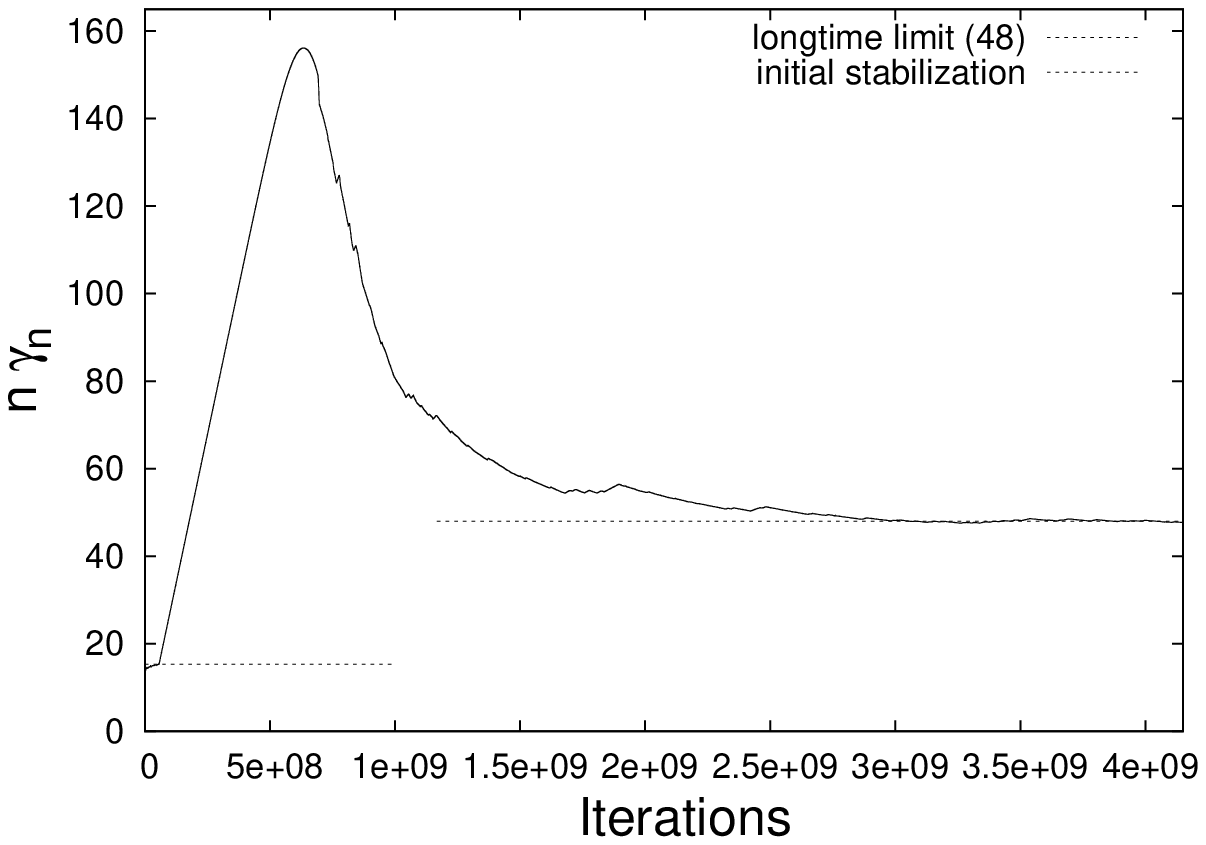}
  \includegraphics[width=0.5\textwidth]{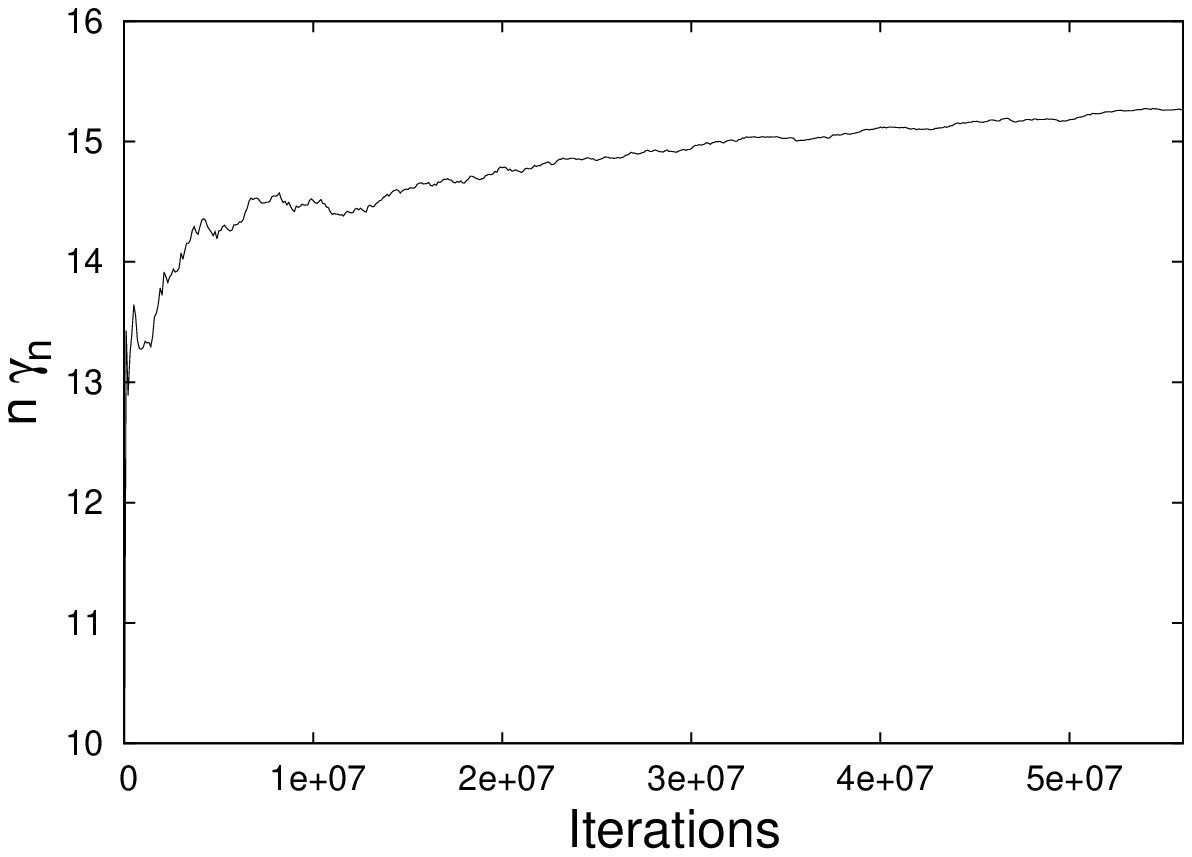}
  \caption{\label{fig:eff_time_step} Behavior of the stepsize
    sequence~$(\gamma_n)_{n \geq 1}$ ($d=48$, $\sigma = 2R/d = 0.05$, $\gamma = 1$ and
    $\beta = 10$). Top: Longtime convergence. Bottom: Stabilization
    during the exploration of the left well before the first
    exit.}
\end{figure}

\subsection{Scalings of the first exit time}
\label{sec:scalings}

In this section, we study the influence of the three parameters
$\sigma$, $\gamma$ and $d$ on the first exit times (in the limit of
small temperature). Concerning $\tu_0$, we stick to the assumption
that without any prior knowledge on the system, the choice
$\tu_0(1)=\ldots=\tu_0(d)$ is natural. In addition, notice that multiplying $\tu_0$ and the parameter
$\gamma$ by the same constant $c>0$ does not modify the sequence $(\tn_n,X_n)_{n
  \geq 0}$ generated by the SHUS algorithm. This is why we always choose $\tu_0(1)=\ldots=\tu_0(d)=1/d$.

Average exit times are obtained by performing independent realizations of the
following procedure: initialize the system in the state $X_0=(-1,0)$, and run
the dynamics until the first time index $\mathcal{N}$ such that
$X_{{\mathcal{N}},1} > 1$ ({\it i.e.} the first component of $X_\mathcal{N}$
is larger than~1) for SHUS or ${X}^{\rm WL}_{{\mathcal{N}},1} > 1$ for
Wang-Landau.  The average of this first exit time is denoted by $t_\beta$ for SHUS and by
${t^{\rm WL}_\beta}$ for Wang-Landau. For a given value of the inverse
temperature~$\beta$, the computed average exit times~$t_\beta$
and~${t^{\rm WL}_\beta}$ are obtained by averaging over $K$ independent
realizations of the process started at $X_0$.  We use the Mersenne-Twister
random number generator as implemented in the GSL library. Since we work with a fixed maximal computational time 
(of about a week or two on our computing machines with our implementation of the code),
$K$ turns out to be of the order of a few thousands for the largest exit times, while $K=10^5$ in the
easiest cases corresponding to the shortest exit times. In our
numerical results, we checked that $K$
is always sufficiently large so
that the relative error on~$t_\beta$ and~${t^{\rm WL}_\beta}$ is less than a
few percents in the worst cases.

According to~the numerical experiments performed in~\cite{amrx} that confirm
the theoretical analysis of a simple three-states model also given in~\cite{amrx}, 
the scaling behavior for $t^{\rm WL}_\beta$ in the limit $\beta \to
\infty$ for the Wang-Landau algorithm with a stepsize sequence $(\gamma_\star/n)_{n\geq 1}$ is
\begin{equation}
  \label{eq:predicted_law_ref}
  {t^{\rm WL}_\beta} \sim C_{\rm WL}\exp(\beta \mu_{\rm WL}) \eqsp,
\end{equation}
where $C_{\rm WL}$ and $\mu_{\rm WL}$ are positive constants which
depend on~$\sigma$, $\gamma_\star$ and $d$.

Due to the initial convergence of $(n\gamma_n)_{n \geq 1}$ to the number $d_{\rm sv}$ of
strata visited before the first exit time, one expects the first exit time of the
SHUS algorithm to behave like the first exit time for the Wang-Landau algorithm with
stepsizes $(d_{\rm sv}/n)_{n\geq 1}$. Note that since we use strata of equal sizes, the number $d_{\rm sv}$ is
proportional to the total number $d$ of strata (see Table~\ref{tab:SHUS_WL} 
for a more quantitative assessment).

We consider three situations:
\begin{enumerate}[(i)]
\item We first study how the exit times vary as a function of $d$ with
  $\sigma = 2R/d$ and the fixed value $\gamma = 1$; see Figure~\ref{fig:SHUS_d} and
  Table~\ref{tab:SHUS_d}.
\item We then fix $\sigma = 0.1$ and study how the exit times depend on $d$, still with the fixed value $\gamma = 1$; see Figure~\ref{fig:SHUS_dx} and Table~\ref{tab:SHUS_dx}.
\item We finally study the scaling of the exit times depending on the
  value $\gamma$ when $d = 12$ and $\sigma = 2R/d = 0.2$ are fixed; see Figure~\ref{fig:SHUS_gamma} and Table~\ref{tab:SHUS_gamma}.  
\end{enumerate}

We observe in all cases that the average first exit time is of the form 
\[
t_\beta \sim C(\gamma,d,\sigma) \, \exp(\beta \mu)
\]
in the limit of large $\beta$, the values $\mu$ and $C(\gamma,d,\sigma)$ being obtained by a least-square fit in log-log scale. In view of~\eqref{eq:predicted_law_ref}, this confirms the validity of the above comparison with the Wang-Landau algorithm. The exponential rate
$\mu$ does not seem to
depend on the values of the parameters $\sigma$, $d$ and $\gamma$.  Only the
prefactors~$C(\gamma,d,\sigma)$ depend on these parameters. The larger the
number of strata, and the lower the value of~$\gamma$, the larger the
prefactor is. A more quantitative assessment of the increase of the prefactor
with respect to larger numbers of strata~$d$ and smaller values of~$\gamma$ is provided in
the captions of Table~\ref{tab:SHUS_dx} and~\ref{tab:SHUS_gamma}.

\begin{figure}
  \includegraphics[width=0.5\textwidth]{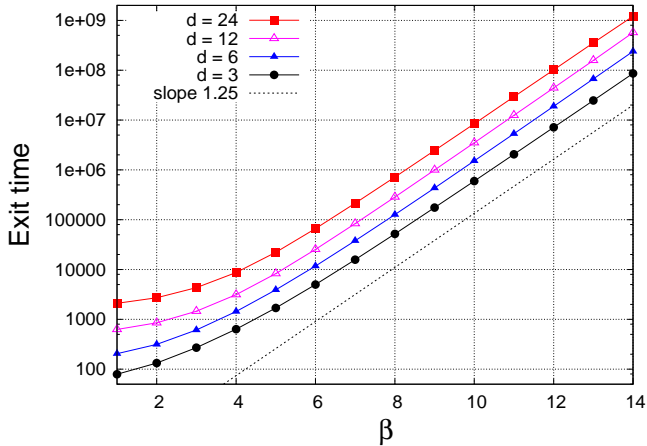}
  \caption{\label{fig:SHUS_d} Exit times as a function of the inverse temperature~$\beta$ when the number of strata $d$ is varied while the magnitude $\sigma$ of the proposed displacement is modified accordingly as $\sigma = 2R/d$ (with $\gamma=1$ fixed).}
\end{figure}

\begin{figure}
  \includegraphics[width=0.5\textwidth]{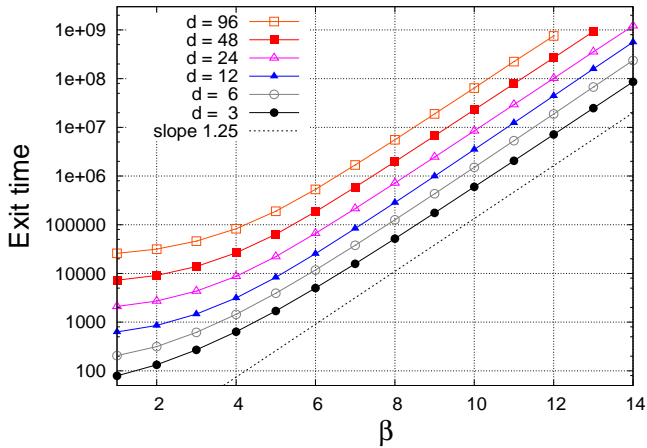}
  \caption{\label{fig:SHUS_dx} Exit times as a function of the inverse temperature~$\beta$ when the number of strata~$d$ is varied, with the fixed values $\gamma = 1$ and $\sigma = 0.1$.}
\end{figure}

\begin{figure}
  \includegraphics[width=0.5\textwidth]{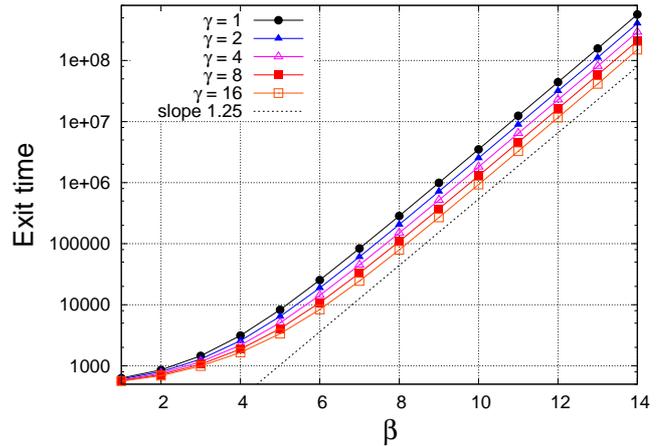}
  \caption{\label{fig:SHUS_gamma} Exit times as a function of the inverse temperature~$\beta$ when $\gamma$ is varied, with the fixed values $d=12$ and $\sigma = 2R/d = 0.2$.}
\end{figure}

\begin{table}
\caption{Scaling law $t_\beta \sim C(\gamma,d,\sigma) \mathrm{e}^{\beta \mu}$ for Figure~\ref{fig:SHUS_d}.}
\label{tab:SHUS_d}      
\begin{tabular}{ccc}
\hline\noalign{\smallskip}
& slope $\mu$ & prefactor $C(\gamma,d,\sigma)$ \\
\noalign{\smallskip}\hline\noalign{\smallskip}
$d=3$  &  1.24 & 2.56 \\
$d=6$  &  1.26 & 5.27 \\
$d=12$ &  1.27 & 11.1 \\
$d=24$ &  1.24 & 34.9 \\
\noalign{\smallskip}\hline
\end{tabular}
\end{table}

\begin{table}
\caption{Scaling law $t_\beta \sim C(\gamma,d,\sigma) \mathrm{e}^{\beta \mu}$ for Figure~\ref{fig:SHUS_dx}. It appears that $C(\gamma,d,\sigma) \simeq C(\gamma,d_0,\sigma) (d/d_0)^{1.4}$ for some reference value~$d_0$.}
\label{tab:SHUS_dx}      
\begin{tabular}{ccc}
\hline\noalign{\smallskip}
& slope $\mu$ & prefactor $C(\gamma,d,\sigma)$ \\
\noalign{\smallskip}\hline\noalign{\smallskip}
$d=3$   &  1.24 & 2.48 \\
$d=6$   &  1.26 & 5.00 \\
$d=12$   &  1.27 & 10.8 \\
$d=24$   &  1.24 & 34.9 \\
$d=48$  &  1.23 & 102 \\
$d=96$ &  1.23 & 295 \\
\noalign{\smallskip}\hline
\end{tabular}
\end{table}

\begin{table}
\caption{Scaling law $t_\beta \sim C(\gamma,d,\sigma) \mathrm{e}^{\beta \mu}$ for Figure~\ref{fig:SHUS_gamma}. It appears that $C(\gamma,d,\sigma) \simeq C(1,d,\sigma) \gamma^{-1/2}$.}
\label{tab:SHUS_gamma}      
\begin{tabular}{ccc}
\hline\noalign{\smallskip}
& slope $\mu$ & prefactor $C(\gamma,d,\sigma)$ \\
\noalign{\smallskip}\hline\noalign{\smallskip}
$\gamma = 1$  & 1.27 & 10.8\\
$\gamma = 2$  & 1.27 & 7.97 \\
$\gamma = 4$  & 1.27 & 5.69 \\
$\gamma = 8$  & 1.27 & 4.01 \\
$\gamma = 16$ & 1.27 & 2.98 \\
\noalign{\smallskip}\hline
\end{tabular}
\end{table}

To conclude these numerical investigations on the comparison between
SHUS and Wang-Landau, we look for which values of
$\gamma_\star$ the two average first exit times ${t^{\rm WL}_\beta}$ and
$t_\beta$ behave similarly (in the
limit of large $\beta$). Some average
exit times for SHUS and the Wang-Landau algorithm are presented in
Figure~\ref{fig:SHUS_WL}. Table~\ref{tab:SHUS_WL} gives intervals of values
for~$\gamma_\star$ for which the exponential rate of increase of the exit times
for the Wang-Landau dynamics matches the ones for the SHUS dynamics. 

As the number of strata is increased, the value of~$\gamma_\star$ has to be
increased in order to retrieve the same asymptotic scaling of exit
times. More precisely, we observe that $\gamma_\star$ should be
proportional to $d$ and this confirms the above comparison between
SHUS and the Wang-Landau algorithm with stepsizes $(d_{\rm sv}/n)_{n\geq 1}$ 
(since $d_{\rm sv}$ is indeed proportional to $d$). It is possible to estimate more precisely the number $d_{\rm sv}$ of
strata visited before the first exit time by a graphical inspection, as illustrated
in Figure~\ref{fig:histo_d_sv}. With this estimate of $d_{\rm sv}$, we
observe that $\gamma_\star$ is indeed very close to $d_{\rm sv}$,
although a bit smaller. 

\begin{figure}
  \includegraphics[width=0.5\textwidth]{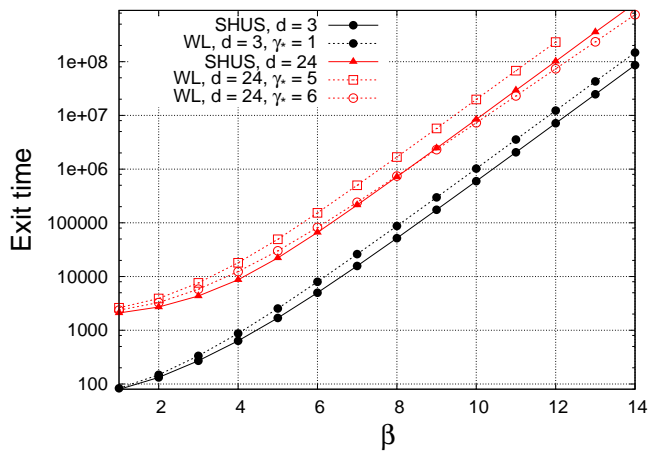}
  \caption{\label{fig:SHUS_WL} Exit times as a function of the inverse temperature~$\beta$, for SHUS and Wang-Landau dynamics, with the choice $\sigma = 2R/d$. In the case $d=24$, the exits times for SHUS are smaller than for Wang-Landau with parameter $\gamma_\star = 5$, but larger than Wang-Landau with parameter $\gamma_* = 6$.}
\end{figure}

\begin{figure}
  \hspace{-0.4cm}\includegraphics[width=0.55\textwidth]{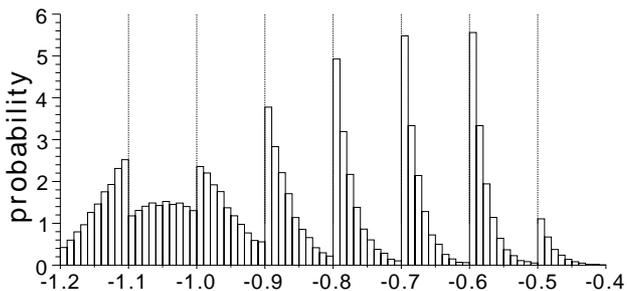}
  \caption{\label{fig:histo_d_sv} Histogram of the values of $x_1$
    visited along a typical trajectory before the first exit of the
    left metastable state, for $\beta = 14$, $d=24$ and $\sigma = 2R/d
    = 0.1$. The frontiers of the strata are indicated by vertical lines. In this example, $d_{\rm sv} = 7-8$ since 7~strata are very well visited, while the stratum corresponding to $-0.5 \leq x_1 \leq -0.4$ is somewhat less visited.}
\end{figure}

\begin{table}
\caption{Values of $\gamma_\star$ for which the asymptotic exponential rate of increase of the exit times for Wang-Landau match the exponential rate of the SHUS dynamics, and values of $d_{\rm sv}$ obtained at $\beta = 14$ by a graphical inspection (see Figure~\ref{fig:histo_d_sv}).}
\label{tab:SHUS_WL}      
\begin{tabular}{ccc}
\hline\noalign{\smallskip}
& $d_{\rm sv}$ & equivalent $\gamma_\star$ \\
\noalign{\smallskip}\hline\noalign{\smallskip}
$d=3$  & 1    & 1          \\
$d=6$  & 2    & 1.6 -- 1.8 \\
$d=12$ & 4    & 2.5 -- 3   \\
$d=24$ & 7 -- 8 & 5 -- 6   \\
\noalign{\smallskip}\hline
\end{tabular}
\end{table}

\section{Discussion, perspectives and extensions}
\label{sec:SHUS_WL}

\subsection{Comparison of SHUS and the Wang-Landau algorithm}

In this section, we would like to summarize our findings about the
comparison between SHUS and the Wang-Landau algorithm.
As explained in Section~\ref{sec:1}, SHUS can be seen as a
Wang-Landau algorithm with the stepsize sequence  $(\gamma_n)_{n
  \geq 1}$ defined by~\eqref{eq:def:stepsize}. From Corollary~\ref{cor:asumptpas},
we thus expect that SHUS behaves in the
longtime regime like the Wang-Landau algorithm with stepsize
$\gamma^{\rm WL}_n= \frac{d}{n}$. These
predictions drawn from our theoretical analysis have been confirmed in
the previous section by numerical experiments. Consistently, it has been shown
that in terms of first exit times from a metastable state, SHUS and the
Wang-Landau algorithm with $\gamma^{\rm WL}_n= \frac{\gamma_\star}{n}$
have similar behaviors, $\gamma_\star$ being close to the number of strata
visited in the metastable state containing the initial condition~$X_0$. 

We have observed numerically that the average exit
time out of a metastable state for the SHUS algorithm is not drastically modified when
changing the numerical parameters $d$, $\sigma$ and $\gamma$ (see
Tables~\ref{tab:SHUS_d},~\ref{tab:SHUS_dx}  and~\ref{tab:SHUS_gamma}
where $\mu$ remains approximately constant). Moreover, it is known
that, in the longtime regime, the choice $\gamma_\star=d$ is the optimal
one for Wang-Landau for stepsize sequences of the form
$\frac{\gamma_\star}{n}$ in terms of asymptotic variance of the weight
sequence, see the discussion after Theorem 3.6 in~\cite{fort:jourdain:kuhn:lelievre:stoltz:2014}. This can be
seen as advantages of SHUS over Wang Landau for which, in particular,
a substantial increase in the exponential
rate of the exit time is observed when $d$ is increased while
$\gamma_\star$ is fixed (see~\cite[Table~1]{amrx}).

On the downside, we observe that the scaling $\gamma^{\rm WL}_n \sim
\frac{\gamma_\star}{n}$ (when $n \to \infty$) is usually not the best one in
terms of efficiency. As explained in~\cite{fort:jourdain:kuhn:lelievre:stoltz:2014}, convergence is also obtained for
larger stepsizes $\gamma^{\rm WL}_n=
\frac{\gamma_\star}{n^\alpha}$ with $\alpha \in (1/2,1)$, which allow
much smaller average exit times from \linebreak metastable states,
see~\cite[Figure 3]{amrx}. In this respect, SHUS is not the most efficient
Wang-Landau type algorithm.

\subsection{Accelerating SHUS}
\label{sec:acc}

In view of the efficiency results of the Wang-Landau dynamics
(see~\cite{amrx} for an analysis of the exit times from metastable states) and
according to general prescriptions for stochastic approximation
algorithms, it seems better to aim for a stepsize sequence
$(\gamma_n)_{n \geq 1}$ which decreases at the slowest possible rate
while still guaranteeing convergence. In the polynomial schedule, this
means that $\gamma_n = \gamma_\star /n^\alpha$ with $\alpha$ close to
$1/2$ rather than $\alpha=1$.
However, it is also known that the asymptotic variance scales as
$\gamma_n$  (see for
example~\cite[Theorem~3.6]{fort:jourdain:kuhn:lelievre:stoltz:2014})
which calls for a very fast decaying stepsize sequence
$(\gamma_n)_{n \geq 1}$ while still guaranteeing convergence. A good practical compromise
is then to combine the stochastic approximation algorithms
with an averaging technique~\cite{fort:2014,polyak-juditsky-92}. This allows to use a large stepsize
sequence (which yields good exploration properties) while keeping a
variance of the optimal order $1/n$.

In view of the above discussion, a natural question is whether SHUS can
be modified in order to obtain an effective stepsize sequence which
scales like $n^{-\alpha}$ with $\alpha \in (1/2,1)$.
A possible way of doing so is to modify the updating rule~\eqref{eq:def:thetatilde} as
 \begin{equation}
 \label{eq:tu_p}
\tu_{n+1}(i) = \tu_n(i)\left(1+\frac{\gamma(\alpha) \un_{\Xset_i}(X_{n+1})}{\ln\left(1+\sum_{j=1}^d\tu_n(j)\right)^{\frac{\alpha}{1-\alpha}}}\right)
\end{equation}
for some positive deterministic $\gamma(\alpha)$. 
We call SHUS$^\alpha$ the algorithm which consists in choosing deterministic $X_0 \in \Xset$ and
$\tu_0=(\tu_0(1),\ldots,\tu_0(d))\in(\R_+^*)^d$ then iterating the following steps:
\begin{algo}\label{algoshusa}
 Given $(\tu_n,X_n)\in
(\R_+^*)^d\times \Xset$,
\begin{itemize}
   \item compute 
     the probability measure on $\{1,\ldots,d\}$,
     \begin{equation*}
    \tn_n = \frac{\tu_n}{\sum_{j=1}^d \tu_n(j)} \in \Theta\eqsp,
     \end{equation*}
   \item draw $X_{n+1}$ according to the kernel $P_{\tn_n}(X_{n},\cdot)$,
   \item compute, for all $i \in \{1, \ldots d\}$, $\tu_{n+1}(i)$ given by \eqref{eq:tu_p}.
  \end{itemize}
\end{algo}SHUS$^\alpha$ can be seen as a Wang-Landau algorithm with
nonlinear update of the weights and with stochastic stepsizes
\begin{equation}\label{eq:gammanSHUSalpha}
\text{for $n\in\N$}, \,
\gamma_{n+1}=\frac{\gamma(\alpha)}{\ln\left(1+\sum_{j=1}^d\tu_n(j)\right)^{\frac{\alpha}{1-\alpha}}}.
\end{equation}

\begin{proposition}\label{propshusa}
  Under A\ref{hyp:targetpi} and A\ref{hyp:kernel}, for each
    $\alpha\in (\frac{1}{2},1)$, the SHUS$^\alpha$ Algorithm~\ref{algoshusa} converges in the sense of Definition 
\ref{defconvalgo}. Moreover its stepsize sequence $(\gamma_n)_{n\geq 1}$ defined by~\eqref{eq:gammanSHUSalpha} satisfies 
\begin{equation}
\label{eq:limit_gamma_n_shus_alpha}
\P\left(
\lim_{n \to \infty} n^\alpha \gamma_n=
\gamma(\alpha)^{1-\alpha}d^\alpha(1-\alpha)^{\alpha}\right)=1.
\end{equation}
\end{proposition}
In particular, the stepsize sequence scales like $n^{-\alpha}$ as wanted.

\begin{remark}[Relationship with the standard SHUS algorithm]
   To relate the updating rule~\eqref{eq:tu_p} with the one of SHUS~\eqref{eq:def:thetatilde}, notice that~\eqref{eq:def:thetatilde} also writes
\begin{align*}
   \tu_{n+1}(i)& = \tu_n(i)\left(1+\frac{\gamma\un_{\Xset_i}(X_{n+1})}{\sum_{j=1}^d\tu_n(j)}\right)\\&=\tu_n(i)\left(1+\frac{\gamma \un_{\Xset_i}(X_{n+1})}{{\dps \lim_{\alpha\to 1^-}}f\left(\alpha,\sum_{j=1}^d\tu_n(j)\right)}\right)
\end{align*}
where for $\alpha\in(1/2,1)$ and
$s>0$, 
\[
f(\alpha,s)=\exp\left(\frac{\alpha}{1-\alpha}\ln\Big[1+(1-\alpha)\ln(1+s)\Big]\right)-1.
\]
Note that, for a fixed $\alpha\in(1/2,1)$ and in the limit $s\to +\infty$,
\[
f(\alpha,s)\sim \Big((1-\alpha)\ln(1+s)\Big)^{\frac{\alpha}{1-\alpha}}.
\]
SHUS$^\alpha$ is therefore expected to have the same asymptotic behavior in the
limit $n\to +\infty$ as the algorithm based on the more complicated updating rule
\begin{equation}\label{eq:tu_p_prime}
  \tu_{n+1}(i) = \tu_n(i)\left(1+\frac{\gamma(\alpha)
    (1-\alpha)^{\frac{\alpha}{1-\alpha}}\,\un_{\Xset_i}(X_{n+1})}{f\left(\alpha,\sum_{j=1}^d\tu_n(j)\right)}\right).
\end{equation}
Notice that, with the choice,
\begin{equation}
  \label{eq:gamma_alpha}
  \gamma(\alpha)=(1-\alpha)^{-\frac{\alpha}{1-\alpha}} \gamma,
\end{equation}
the updating rule~\eqref{eq:tu_p_prime}
converges to the SHUS updating rule~\eqref{eq:def:thetatilde} when
$\alpha \to 1^-$. In addition, from Proposition~\ref{propshusa},
$\lim_{n \to \infty} n^\alpha \gamma_n=
\gamma^{1-\alpha}d^\alpha$ almost surely. The limiting value $\gamma^{1-\alpha}d^\alpha$ converges to $d$ when
$\alpha \to 1^-$  which is also consistent with what we obtained for the
original SHUS algorithm, see Corollary~\ref{cor:asumptpas}.
\end{remark}

\subsubsection*{Numerical results}

We present in this section numerical results showing that the modified SHUS algorithm (Algorithm~\ref{algoshusa}) with parameter $\alpha \in (1/2,1)$ behaves very similarly to the Wang-Landau algorithm with stepsizes scaling as~$n^{-\alpha}$ as $n \to +\infty$. We choose $\gamma(\alpha)$ according to~\eqref{eq:gamma_alpha} in all cases.

In order to implement the modified SHUS algorithm, some care has to be taken in order to avoid overflows related to large values of the normalization factor \linebreak $\sum_{i=1}^d \tu_n(i)$. Prohibitively large numbers can be avoided by first representing the weighted occupation factors in logarithmic scale as $\nu_n(i) = \ln \tu_n(i)$. Second, in order to avoid the uncontrolled increase of $\nu_n(i)$ (which may be quite fast for large values of~$\beta$ and values of~$\alpha$ close to $1/2$), we renormalize the factors at random stopping times where $\sum_{i=1}^d \tu_n(i)$ is greater than a given (large) value $M > 0$. The weight sequence renormalized at these random stopping times is denoted by $\nu_n^M(i)$. The random stopping times are defined as $\tau_0 = 0$, and
\[
\tau_k = \inf \left\{ n > \tau_{k-1} \, : \, \sum_{i=1}^d \mathrm{e}^{\nu_n^M(i)} \geq M \right\} \qquad (k \geq 1).
\]
The renormalized factors $\nu_n^M(i)$ evolve according to the following updating rule (obtained by taking the logarithm of~\eqref{eq:tu_p}, and possibly subtracting a renormalization factor):
\[
\nu_{n+1}^M(i) = \nu_n^M(i) + \ln\left(1+\gamma^M_{n+1} \un_{\Xset_i}(X_{n+1}) \right)- \sigma_n \ln M.
\]
In this expression, $\sigma_n = 0$ if $n \neq \tau_1,\dots,\tau_k,\dots$ while $\sigma_n = 1$ if there exists $k \geq 1$ such that $\tau_k = n$. In addition, the stepsize is
\[
\gamma_{n+1}^M = \frac{\gamma(\alpha) }{\left[ \ln\left(M^{-r_n} + \sum_{j=1}^d\mathrm{e}^{\nu^M_n(j)}\right) + r_n \ln M \right]^{\frac{\alpha}{1-\alpha}}},
\]
where $r_n = \sum_{m=1}^n \sigma_m$ counts the number of times where the weights have been renormalized up to the iteration index~$n$. The logarithmic normalized weights $\ln \tn_n(i)$ are then constructed from the renormalized occupation measures in logarithmic scale $\nu_n^M(i)$ as
\[
\ln \tn_n(i) = \nu_n^M(i) -\ln \left(\sum_{j=1}^d \mathrm{e}^{\nu_n^M(j)}\right).
\]
The sequence of visited states $(X_n)_{n \geq 1}$ and the sequence of weights $(\tn_n)_{n \geq 1}$ do not depend on the value of~$M$, as long as $M$ is in the range of number which can be represented on a computer. In the simulations reported below, we chose $M = 10^{10}$. In addition, in order to have a well-behaved acceptance procedure in the Metropolis step, we compute the probability to accept the proposed move in logarithmic scale. More precisely, the proposed move $\widetilde{X}_{n+1}$ drawn from the proposal kernel starting from~$X_n$ is accepted when
\[
\begin{aligned}
\ln U^n & \leq \ln \pi\left(\widetilde{X}_{n+1}\right) - \ln \pi(X_n) \\
& \quad - \nu_n^M\left( I\left(\widetilde{X}_{n+1}\right) \right) + \nu_n^M\left( I\left(X_{n}\right) \right),
\end{aligned}
\]
where $U^n$ is a sequence of i.i.d. uniform random variables on~$[0,1]$.

We first check whether the effective stepsizes $\gamma_n$ behave as predicted by Proposition~\ref{propshusa}. To this end, we simulate the modified SHUS algorithm, using increments sampled according to isotropic independent Gaussian random variables with variance~$\sigma^2$ to generate proposal moves in the Metropolis-Hastings algorithm. 
The value of the renormalized stepsizes
\[
\frac{n^\alpha \gamma_n}{d^\alpha (1-\alpha)^\alpha\gamma(\alpha)^{1-\alpha}} = \left(\frac{n}{d}\right)^\alpha \frac{\gamma_n}{\gamma}
\]
are plotted in Figure~\ref{fig:eff_time_step_mod_SHUS}. As expected, the longtime limit is~1 whatever the value of~$\alpha$. However, the larger~$\alpha$ is, the longer it takes to attain the asymptotic regime. This is due to the fact that exit times out from the left well of the potential are typically increasing as $\alpha$ is increased. 

We study more precisely the behavior of the exit times $t_\beta$ out from the left well in the limit of low temperatures (large $\beta$), following the procedure described in Section~\ref{sec:scalings}, for various values of~$\alpha \in (1/2,1)$. The data presented in Figure~\ref{fig:exit_time_mod_SHUS} have been fitted to power laws
\begin{equation}
\label{eq:predicted_scaling_mod_shus}
t_\beta \sim C_\alpha t^{\mu_\alpha}.
\end{equation}
Such a power law scaling indeed accounts for the behavior of exit times for the standard Wang-Landau algorithm with stepsizes $\gamma_n^{\rm WL} = \frac{\gamma_\star}{n^\alpha}$, and it can be proved that
\begin{equation}
\label{eq:theoretical_mu_alpha}
\mu_\alpha = \frac{1}{1-\alpha}
\end{equation}
for a simple model system (see~\cite{amrx}). This claim is also backed up in~\cite{amrx} by numerical experiments on the same model as the one considered in this work. The powers $\mu_\alpha$ reported in Table~\ref{tab:mod_SHUS}, obtained by a least-square minimization in log-log scale, agree very well with the theoretical prediction~\eqref{eq:theoretical_mu_alpha}. This indicates that SHUS$^\alpha$ is very close to Wang-Landau with stepsizes $\gamma_n = \gamma/n^\alpha$.

\begin{figure}
  \includegraphics[width=0.5\textwidth]{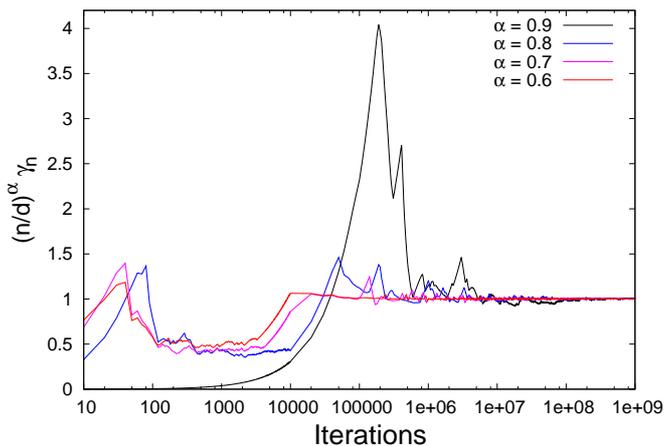}
  \caption{\label{fig:eff_time_step_mod_SHUS} Behavior of the sequence~$((n/d)^\alpha \gamma_n)_{n \geq 1}$ for $d=24$, $\sigma = 2R/d = 0.1$, $\gamma = 1$ and $\beta = 10$, for various values of~$\alpha$. The longtime limit is in all cases~1, as predicted by Proposition~\ref{propshusa}.
  }
\end{figure}

\begin{figure}
  \includegraphics[width=0.5\textwidth]{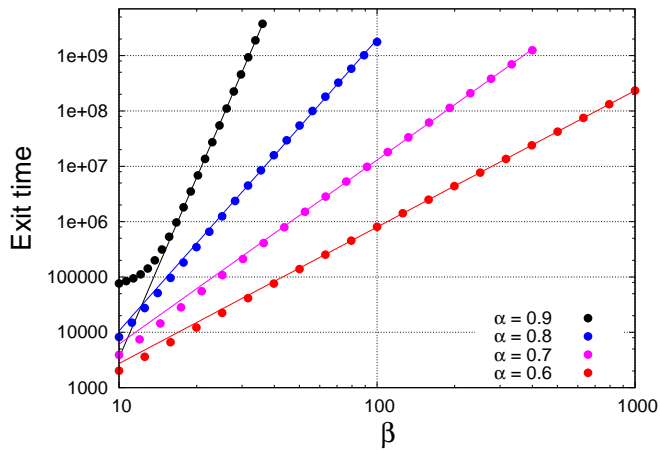}
  \caption{\label{fig:exit_time_mod_SHUS} 
    Computed exit times $t_\beta$ for the modified SHUS algorithm with various choices of the power~$\alpha$, and the parameters $d=24$, $\sigma = 2R/d = 0.1$, $\gamma = 1$. A power law $t_\beta \sim C_\alpha \beta^{\mu_\alpha}$ (plotted in solid lines) is observed in all cases. Estimated powers are reported in Table~\ref{tab:mod_SHUS}.}
\end{figure}

\begin{table}
\caption{Scaling laws~\eqref{eq:predicted_scaling_mod_shus} obtained from the data presented in Figure~\ref{fig:exit_time_mod_SHUS}.}
\label{tab:mod_SHUS}      
\begin{tabular}{ccc}
\hline\noalign{\smallskip}
$\alpha$ & predicted value $(1-\alpha)^{-1}$ & numerical fit $\mu_\alpha$\\
\noalign{\smallskip}\hline\noalign{\smallskip}
0.6  &  2.50 & 2.47 \\
0.7  &  3.33 & 3.33 \\
0.8  &  5    & 5.27 \\
0.9  &  10   & 10.8 \\
\noalign{\smallskip}\hline
\end{tabular}
\end{table}

We next study the convergence of the logarithmic weights $\ln \tn_n(i)$. The empirical variance of the logarithmic weights at iteration~$n$ is estimated by running $K$ independent realizations of the modified SHUS dynamics, as for the estimation of exit times. Denoting by $(\ln \tn_n^k(i))_{n \geq 1}$ the weight sequence of the $i$th stratum for the $k$th realization, the empirical variance of the weight for the stratum~$i$ is estimated as
\[
\mathscr{V}_{n,K}(i) = \frac{1}{K-1} \sum_{k=1}^K \Big( \ln \tn_n^k(i) - \mathscr{M}_{n,K}(i)\Big)^2, 
\]
where the empirical mean $\mathscr{M}_{n,K}(i)$ at time~$n$ is
\[
\mathscr{M}_{n,K}(i) = \frac{1}{K} \sum_{k=1}^K \ln \tn_n^k(i).
\]
We expect that $\mathscr{V}_{n,K}(i)$ scales as $\gamma_n$ in the longtime limit for $K$ sufficiently large. Such a result is proved for the Wang-Landau dynamics in~\cite{fort:jourdain:kuhn:lelievre:stoltz:2014}, where a central limit theorem is established for the weight sequence. In the sequel, we take $K = 7 \times 10^4$.

In all the computations reported below, the value $\alpha =1$ corresponds to the standard SHUS algorithm introduced in Section~\ref{sec:1}. The decrease of the empirical variance as a function of time is plotted in Figure~\ref{fig:sample_var}, together with a numerical fit $C_{{\rm var},i} n^{-a_i}$ obtained by a least-square minimization in log-log scale. The decay exponents~$a_i$ for each stratum~$i$ are reported in Figure~\ref{fig:slopes_var}. We indeed confirm that the empirical variance decreases as $n^{-\alpha}$, except in the case $\alpha = 0.9$ where the decrease is slightly faster than expected since the exponents $a_i$ are around~0.95. Note also that the asymptotic regime is attained only at longer times for the value $\alpha = 0.9$. This is related to the fact that $\gamma(\alpha)$ becomes very large for $\alpha$ close to~1 (here $\gamma(0.9) = 10^9$).

\begin{figure}
  \includegraphics[width=0.5\textwidth]{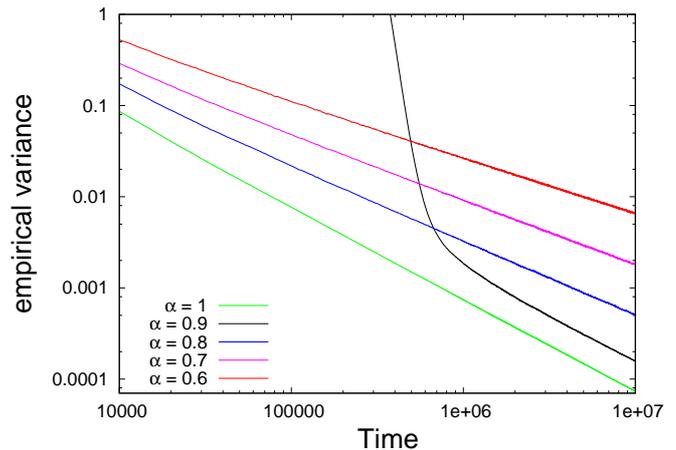}
  \caption{\label{fig:sample_var} 
    Decrease of the empirical variance~$\mathscr{V}_{n,K}(i)$ as a function of the time~$n$, for $d=24$, $\beta = 1$, $\sigma = 2R/d = 0.1$ and $i=3$. The value $\alpha =1$ corresponds to the standard SHUS algorithm introduced in Section~\ref{sec:1}.}
\end{figure}

\begin{figure}
  \includegraphics[width=0.5\textwidth]{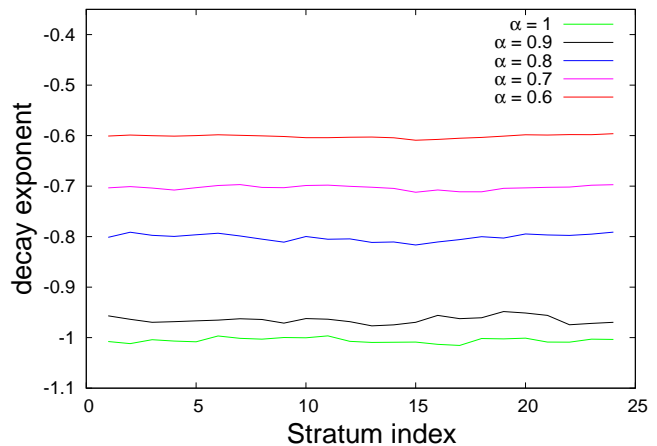}
  \caption{\label{fig:slopes_var} 
    Decay exponents $a_i$ in the fit of the empirical variances $\mathscr{V}_{n,K}(i) \sim C_{{\rm var},i} n^{-a_i}$, for various values of~$\alpha$ (same parameters as in Figure~\ref{fig:sample_var}). }
\end{figure}

We finally consider the decrease of the bias in the estimated empirical averages $\mathscr{M}_{n,K}(i)$, as a function of time. We use a normalized version of the bias and average over the strata, and therefore introduce
\[
\mathscr{B}_{n,K} = \sqrt{\sum_{i=1}^d \left(\frac{\mathscr{M}_{n,K}(i)}{\ln \theta_*(i)}-1\right)^2}.
\]
The reference values $\theta_*(i)$ are computed with a two-dimensional numerical quadrature. The decrease of the bias as a function of time is plotted in Figure~\ref{fig:error_bias}, together with a numerical fit $C_{\rm bias} n^{-a}$ obtained by a least-square minimization in log-log scale. The decay exponents~$a$ are reported in Table~\ref{tab:error_bias}. The bias approximately decrease at the same rate as the variance, namely $n^{-\alpha}$, a standard behavior for Monte-Carlo methods. Here again, the asymptotic behavior for the value $\alpha=0.9$ is observed at longer times only because of the large value of $\gamma(\alpha)$.

\begin{figure}
  \includegraphics[width=0.5\textwidth]{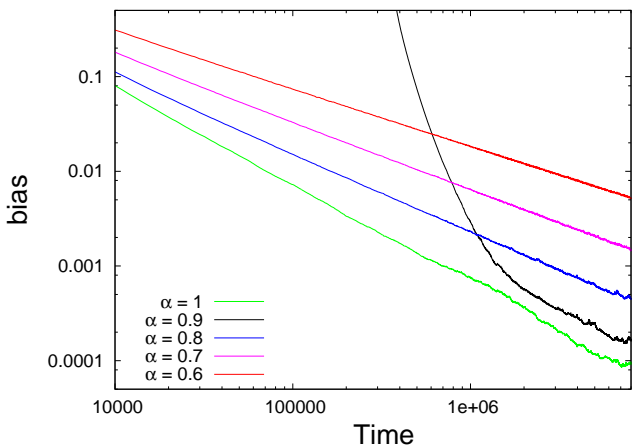}
  \caption{\label{fig:error_bias} Decay of the bias $\mathscr{B}_{n,K}$ as a function of the iteration time~$n$, for $d=24$, $\beta = 1$ and $\sigma = 2R/d = 0.1$. The corresponding decay exponents are reported in Table~\ref{tab:error_bias}.
  }
\end{figure}

\begin{table}
\caption{Decay of the bias $\mathscr{B}_{n,K} \sim C_{\rm bias} n^{-a}$ fitted on the data presented in Figure~\ref{fig:error_bias} for iterations times~$n$ in the range $2 \times 10^6 \leq n \leq 8 \times 10^6$. }
\label{tab:error_bias}      
\begin{tabular}{cc}
\hline\noalign{\smallskip}
$\alpha$ & numerical fit $a$\\
\noalign{\smallskip}\hline\noalign{\smallskip}
0.6  & 0.60  \\
0.7  & 0.69  \\
0.8  & 0.78  \\
0.9  & 0.93  \\
1 & 1.01  \\
\noalign{\smallskip}\hline
\end{tabular}
\end{table}

\subsection{Partially biased dynamics}

It is possible to consider more general biased measures than $\pi_\tn$ defined in~\eqref{eq:def:PiTheta} by applying only a fraction of the bias. This amounts to considering the following biased densities for a given parameter $a \in (0,1]$:
\begin{equation}
  \label{eq:def:PiTheta_alpha}
 \pi_{\tn,a}(x) =  \left( \sum_{j=1}^d \frac{\tstar(j)}{\tn(j)^a}
 \right)^{-1}  \sum_{i=1}^d  \frac{\pi(x)}{\tn(i)^a} \un_{\Xset_i}(x)
 \eqsp.
\end{equation}
The motivation for applying only a fraction of the bias is to avoid
having a random walk among the strata in the asymptotic regime, in
order to
favor the strata which are more likely under the original, unbiased
measure. This idea was first proposed within the so-called
well-tempered metadynamics method\footnote{With the notation of the
  metadynamics works, what we call here $a$ is denoted $\Delta T/(T +
  \Delta T)$ where $T$ is the temperature and $\Delta T > 0$ is a parameter. The limiting regime $a = 1$ is recovered in the limit $\Delta T \to +\infty$, which corresponds to the standard metadynamics~\cite{laio-parrinello-02,bussi-laio-parrinello-06}.}, introduced in~\cite{barducci-bussi-parrinello-08}. 

Following the reasoning of Section~\ref{secdesal}, Algorithm~\ref{algoshus} should then be modified as follows: 
\begin{algo}Given $(\tu_n,X_n)\in (\R_+^*)^d\times \Xset$,\label{algowtmd}
   \begin{itemize}
   \item compute  the probability measure on $\{1,\ldots,d\}$,
     \begin{equation*}
       \tn_n = \frac{\tu_n}{\sum_{j=1}^d \tu_n(j)} \in \Theta\eqsp,
     \end{equation*}
   \item draw $X_{n+1}$ according to the kernel
     $P_{\tn_n,a}(X_{n},\cdot)$ where, for any $\tn \in \Theta$, $P_{\tn,a}$ is a
     transition kernel ergodic with respect to~$\pi_{\tn,a}$ ,
   \item compute, for all $i \in \{1, \ldots d\}$,
\begin{equation}
  \label{eq:update_tu_n_alpha}
  \tu_{n+1}(i) = \tu_n(i) + \pas \,  \tn_n(i)^a \,
  \un_{\Xset_i}(X_{n+1}) \eqsp .
\end{equation}
\end{itemize}
\end{algo}

It is possible to follow the same reasoning as in
Section~\ref{sec:asymptotic_consistency} to check that the only
possible limit for the sequence $(\tn_n)_{n \geq 0}$ is $\theta_\star$.
In addition, $\gamma_n$ (defined
by~\eqref{eq:def:stepsize}) should scale as $1/n$ in the longtime
limit.
However, the stepsize sequence $(\gamma_{n})_{n \geq 1}$ needed to rewrite the updating rule \eqref{eq:update_tu_n_alpha} as a particular case of \eqref{eq:def:thetatildegen} is not predictable since $\gamma_n$ depends on~$X_n$:
\[
\gamma_n = \frac{\theta^a_{n-1}(I(X_{n}))}{\theta_{n-1}(I(X_{n}))} \, .
\] 
The convergence of this new algorithm therefore does not enter the framework of Theorem~\ref{theo:cvggal}.

\begin{remark}
The above algorithm with the modified
update~\eqref{eq:update_tu_n_alpha} is very much related to the
well-tempered metadynamics algorithm~\cite{barducci-bussi-parrinello-08}. The
main difference is that we here consider a discrete reaction
coordinate (associated with a partition of the state space) whereas
the standard well-tempered metadynamics method is formulated for continuous
reaction coordinates. As a matter of fact, in the
paper~\cite{dama-parrinello-voth-14}, the authors made the observation
that the well-tempered metadynamics method is a stochastic approximation
algorithm with a stepsize sequence of order $1/n$ (see in
particular~\cite[Equation~(5)]{dama-parrinello-voth-14}). 

The well-tempered metadynamics algorithm is a \linebreak ``parameter-free'' version of the
original metadynamics algorithm~\cite{laio-parrinello-02,bussi-laio-parrinello-06}. The
original metadynamics algorithm consists in penalizing already visited states in a fashion very similar
to the Wang-Landau algorithm. In particular, the original
metadynamics also requires to choose a vanishing stepsize sequence, in
order to penalize less and less the visited states as time goes. One
of the reason why the well-tempered metadynamics algorithm has then been
proposed is to avoid the choice of this sequence. In the well-tempered
dynamics method (as in the SHUS dynamics), the penalization decreases as the
bias of the sampled point becomes larger. Roughly speaking, SHUS is a
parameter-free version of Wang-Landau, in the same way as
well-tempered metadynamics algorithm is a parameter-free version of metadynamics.
As explained above, the
parameter-free version corresponds to a specific choice of the
stepsize sequence, with a $1/n$ scaling for the strength of the
penalization which does not seem to be the optimal choice in terms
of exit times from metastable states, as discussed in Section~\ref{sec:acc}.
\end{remark}


\section{Proofs}
\label{sec:proofs}
\subsection{Proof of Proposition \ref{prop:stabilite}}
For the
  Wang-Landau algorithm with linear update of the weights
  \eqref{eq:linear_update} and deterministic non-increasing stepsize sequence,
  \eqref{proprec} is proved in~\cite[Section
  4.2]{fort:jourdain:kuhn:lelievre:stoltz:2014}. It is explained
  in~\cite[Section~4.2.4]{fort:jourdain:kuhn:lelievre:stoltz:2014} how to adapt
  the proof to the Wang-Landau algorithm with nonlinear update of the weights
  and deterministic non-increasing stepsize sequence. In addition, a careful look at the
  arguments in \cite{fort:jourdain:kuhn:lelievre:stoltz:2014} shows that the
  randomness of the sequence $(\gamma_n)_{n \geq 1}$ plays no role in the proof of \eqref{proprec} as long as
  the conditional distribution of $X_{n+1}$ given ${\cal F}_n$ is given by
  $P_{\tn_n}(X_n,.)$ and the sequence $(\gamma_n)_{n \geq 1}$ is bounded from above by a deterministic sequence converging to $0$. Hence \eqref{proprec} as well as the existence (proved at the end of~\cite[Section
  4.2.1]{fort:jourdain:kuhn:lelievre:stoltz:2014} for the Wang-Landau algorithm with linear update) of a deterministic $p\in(0,1)$ such that \begin{equation}
    \label{eq:recomp_1}
    \forall k,m\in\N, \quad \P\Big(T_{k+1}-T_k>md \, \Big| \, {\cal F}_{T_k}\Big)\leq (1-p)^m  \eqsp,
  \end{equation}
still hold in the present framework.
Let us deduce from~\eqref{eq:recomp_1} that it is possible to couple a
  sequence $(\tilde{T}_k)_{k\geq 0}$ with the same law as $(T_k)_{k\geq 0}$
  with a sequence $(\tau_k)_{k\geq 1}$ of independent geometric random
  variables with parameter~$p$ in such a way that a.s. 
\begin{equation}\label{eq:Ttau}
\forall k\in\N,\,
  \tilde{T}_{k+1}-\tilde{T}_{k}\leq d\tau_{k+1}.
\end{equation} We can set
  $\tau_k=F^{-1}(U_k)$ where $F^{-1}$ denotes the c\`ag pseudo-inverse of the
  cumulative distribution function $F(x)=\un_{\{x\geq
    0\}}\left(1-(1-p)^{\lfloor x\rfloor}\right)$ of the geometric law with
  parameter $p$ and $(U_k)_{k\geq 1}$ is a sequence of independent uniform random
  variables on $[0,1]$. Now, define
  \[
  F_{(T_0,\ldots,T_k)}(x) = \P\left(\left.\frac{T_{k+1}-T_k}{d}\leq x\right|(T_0,\ldots,T_k)\right)
  \]
  for $k\geq 0$.  Since the random vector $(T_0,\ldots,T_k)$ is ${\mathcal
    F}_{T_k}$-measurable, for any $x \geq 0$, 
\begin{align*}
  F_{(T_0,\ldots,T_k)}(x)&\geq F_{(T_0,\ldots,T_k)}(\lfloor x\rfloor)
  \\
  & \hspace{-1.5cm} = 1-\E\Big[\P\big(T_{k+1}-T_k>\lfloor x\rfloor d \, \big|\,{\cal F}_{T_k}\big)\Big|(T_0,\ldots,T_k)\Big]\\ 
  & \hspace{-1.5cm} \geq
  1-(1-p)^{\lfloor x\rfloor}=F(x). \end{align*} where we used
\eqref{eq:recomp_1} for the second inequality. The sequence
$(\tilde{T}_k)_{k\geq 0}$ defined inductively by $\tilde{T}_0=0$ and
\[
\forall k\in\N, \quad \tilde{T}_{k+1}=\tilde{T}_k+d\,F^{-1}_{(\tilde{T}_0,\ldots,\tilde{T}_k)}(U_{k+1}),
\]
satisfies the required properties: it has the same law as $(T_k)_{k
  \geq 0}$ and it satisfies~\eqref{eq:Ttau}.

As a consequence,
\begin{align*}
  \P\left( \limsup_{k\to\infty}\frac{T_k}{k}\leq \frac{d}{p}\right) &=\P\left(
    \limsup_{k\to\infty}\frac{\tilde{T}_k}{k}\leq \frac{d}{p}\right) \\
  & \leq \P\left(\limsup_{k\to\infty}\frac{1}{k}\sum_{j=1}^k \tau_j\leq
    \frac{1}{p}\right)=1 \eqsp,
\end{align*}
by the strong law of large numbers for i.i.d.  random variables. One then easily concludes that \eqref{lgtk} holds.
\subsection{Proof of Theorem \ref{theo:cvggal}}
The proof of Theorem~\ref{theo:cvggal} is performed by extending the
  technique of proof used in~\cite{fort:jourdain:kuhn:lelievre:stoltz:2014} for
  the convergence of Wang-Landau. We therefore mention only the needed
  extensions, the main difference
  with~\cite{fort:jourdain:kuhn:lelievre:stoltz:2014} being the fact that the
  stepsizes $\gamma_n$ are not necessarily deterministic.

\medskip 

Note first that Lemma~4.6, Lemma~4.7 and Lemma~4.9 in
\cite{fort:jourdain:kuhn:lelievre:stoltz:2014} remain valid since they only
depend on the expression of $\pi_\tn$, $P_\tn$. Let us prove successively the three items in Definition \ref{defconvalgo}.

\textit{(i)} The proof of the first item consists in verifying the
sufficient conditions given in
\cite[Theorems~2.2. and~2.3]{andrieu:moulines:priouret:2005} for the
convergence of SA algorithms.
Define the mean field function $h: \Theta \to \R^d$ by
\[
h(\tn) = \int_{\Xset} H(x,\tn) \, \pi_\tn( \rmd x) =\frac{\tstar-\theta}{\sum_{i=1}^d\frac{\tstar(i)}{\theta(i)}}\eqsp.
\]
The function $h$ is continuous on $\Theta$. 
By~\cite[Proposition~4.5]{fort:jourdain:kuhn:lelievre:stoltz:2014}, 
the function $V$ defined on $\Theta$ by
\[
V(\tn) = - \sum_{i=1}^d \tstar(i) \ \ln\left(
  \frac{\tn(i)}{\tstar(i)}\right)
\]
is non negative, continuously differentiable on $\Theta$ and the level set
$\{\tn \in \Theta: V(\tn) \leq M \}$ is a compact subset of the open set $\Theta$ for
any $M>0$. We also have $\left<\nabla V(\tn),h(\tn)\right> \leq 0$ and
$\left<\nabla V(\tn),h(\tn)\right> = 0$ if and only if $\tn= \tstar$.
Hence, the assumption A1 of \cite{andrieu:moulines:priouret:2005} is verified
with $\mathcal{L} = \{\tstar \}$. 

Under our assumptions, the conditions on the stepsize sequence $(\gamma_n)_{n \geq 1}$ in~\cite[Theorems~2.2 and~2.3]{andrieu:moulines:priouret:2005} 
 hold almost-surely. To apply these theorems, it is enough to prove that 
\begin{equation}\label{lastcond}
   \lim_k \sup_{ \ell \geq k} \left| \sum_{n=k}^{\ell} \gamma_{n+1} \Big(
    H(X_{n+1},\tn_n) - h(\tn_n) + \Lambda_{n+1} \Big) \right| =0
\end{equation}
with probability one, where $\Lambda_{n+1}(i)$ is defined in~\eqref{eq:def:Lamndan}. Indeed \eqref{proprec}, \eqref{lastcond} and \cite[Theorems~2.2]{andrieu:moulines:priouret:2005} imply that Algorithm~\ref{algogen} is stable in the sense that a.s., the sequence $(\tn_n)_n$ remains in a compact subset of $\Theta$. Then \cite[Theorems~2.3]{andrieu:moulines:priouret:2005} ensures its a.s convergence to $\tstar$. 

Let us now check \eqref{lastcond}.
Since $0 \leq \tn_n(i) \leq 1$, it holds
$|\Lambda_{n+1}| \leq \gamma_{n+1}$.  Hence,
\[
\P\left( \forall k, \sup_{\ell \geq k} \left|\sum_{n=k}^\ell \gamma_{n+1}
    \Lambda_{n+1} \right|\leq  \ \sum_{n \geq k}\gamma_{n+1}^2 \right) =1
\eqsp
\]
so that (\ref{eq:stepsize:assumptions}) implies that $\sup_{\ell \geq k} \left|\sum_{n=k}^\ell \gamma_{n+1}
    \Lambda_{n+1} \right|$ converges to $0$ a.s. as $k\to\infty$.

To deal with $H(X_{n+1},\tn_n) - h(\tn_n)$, for each $\theta\in\Theta$, we introduce the Poisson equation 
\[
\forall x\in\Xset, \quad g(x) - P_\tn g(x) = H(x,\tn) - h(\tn)
\]
whose unknown is the function $g:\Xset\to\R$. 
Under the stated assumptions, this equation admits a solution $\hatH_\tn(x)$ which is unique up to an
additive constant and it holds (see {\it e.g.}~\cite[Lemma 4.9.]{fort:jourdain:kuhn:lelievre:stoltz:2014})
\begin{equation}
  \label{eq:bound:hatH}
  \sup_{\tn \in \Theta} \sup_{x \in \Xset} \left|\hatH_\tn(x)\right| < \infty \eqsp.
\end{equation}
We write \begin{align*} H(X_{n+1},\tn_n) - h(\tn_n) &=
  \hatH_{\tn_n}(X_{n+1})-P_{\tn_n}\hatH_{\tn_n}(X_{n+1})\\&=\mathcal{E}_{n+1} +
  R_{n+1}^{(1)} +R_{n+1}^{(2)},
\end{align*} with
\begin{align*}
  \mathcal{E}_{n+1} & =  \hatH_{\tn_n}(X_{n+1}) - P_{\tn_n}\hatH_{\tn_n}(X_{n})  \eqsp, \\
  R_{n+1}^{(1)} & = P_{\tn_n}\hatH_{\tn_n}(X_{n}) - P_{\tn_{n+1}}\hatH_{\tn_{n+1}}(X_{n+1}) \eqsp, \\
  R_{n+1}^{(2)} & = P_{\tn_{n+1}}\hatH_{\tn_{n+1}}(X_{n+1}) -
  P_{\tn_{n}}\hatH_{\tn_{n}}(X_{n+1}) \eqsp.
\end{align*}
Let us first check, using the ${\cal F}_n$-predictability of the sequence
$(\gamma_n)_{n \geq 1}$, that $M_k = \un_{\{k\geq 1\}}
\sum_{n=1}^k\gamma_n\mathcal{E}_n$ converges a.s. as $k\to\infty$, which will
imply that a.s.
\[
\lim_k \sup_{\ell \geq k} \left|\sum_{n=k}^\ell \gamma_{n+1} \mathcal{E}_{n+1}
\right|= 0 \ .
\]
Since $(\mathcal{E}_{n})_{n \geq 0}$ is bounded by $\sup_{\tn \in \Theta} \sup_{x \in
  \Xset} |\hatH_\tn(x)|$ and $(\gamma_n)_{n \geq 1}$ is bounded by $\overline{\gamma}_1$, for each $k$, 
\[
|M_k|\leq 2k\overline{\gamma}_1\sup_{\tn \in \Theta} \sup_{x \in \Xset} \left|\hatH_\tn(x)\right|
\]
and $M_k$ is square integrable by \eqref{eq:bound:hatH}.
Moreover, $\gamma_{n+1}$ is ${\cal F}_n$-measurable and the conditional
distribution of $X_{n+1}$ given ${\cal F}_n$ is $P_{\tn_n}(X_n,\cdot)$, so that
\begin{multline*}
  \E(\gamma_{n+1}\mathcal{E}_{n+1} \,\Big| \, {\cal F}_n) \\
  =\gamma_{n+1}\left[ \E\left( \left. \hatH_{\tn_n}(X_{n+1}) \right|{\cal F}_n\right) -
  P_{\tn_n}\hatH_{\tn_n}(X_{n}) \right]=0
\end{multline*}
In conclusion, $(M_k)_{k \geq 1}$ is a square integrable ${\cal
  F}_k$-martingale.  Since 
\[
\sum_{n}\E((M_{n+1}-M_n)^2|{\cal F}_n)=\sum_{n
}\gamma_{n+1}^2\E(\mathcal{E}_{n+1}^2|{\cal F}_n)
\]
is smaller than
$C\sum_{n\geq 1}\gamma_n^2$ by (\ref{eq:bound:hatH}) and therefore finite with
probability one by (\ref{eq:stepsize:assumptions}), $(M_k)_{k \geq 1}$
converges a.s. by \cite[Theorem 2.15]{hall:heyde:1980}.

We now consider the term $R_{n+1}^{(1)}$. By (\ref{eq:bound:hatH}) and the
monotonic property of $(\gamma_n)_{n \geq 1}$, following the same lines as in
the proof of \cite[Proposition 4.10]{fort:jourdain:kuhn:lelievre:stoltz:2014} we prove that there exists a
constant $C$ such that 
\begin{align*}
\P\left( \forall k,   \sup_{\ell \geq k} \left|\sum_{n=k}^\ell \gamma_{n+1} R_{n+1}^{(1)}
  \right|\leq C \ \gamma_{k+1} \right) =1 \eqsp. 
\end{align*}
Therefore, $\sup_{\ell \geq k} \left|\sum_{n=k}^\ell \gamma_{n+1} R_{n+1}^{(1)}
\right|$ tends to zero a.s. as $k \to \infty$.

We now consider the term
$R_{n+1}^{(2)}$. We have $\hatH_\tn(x) = \sum_n
P_\tn^n(H(\cdot,\tn)-h(\tn))(x)$; by
\cite[Lemma~4.2]{fort:moulines:priouret:2011}, there exists a constant
$C$ which does not depend on $\tn,\tn'$ (thanks to
Proposition~\ref{prop:ergounif}  and the upper bound \linebreak
$\sup_\theta \sup_x |H(\theta,x)|< \infty$) such
that for any $\tn, \tn' \in \Theta$,
\begin{multline*}
  \sup_\Xset \left| P_{\tn}\hatH_{\tn} - P_{\tn'}\hatH_{\tn'} \right| \leq C
  \left(
    \sup_\Xset \left| H(\cdot,\tn) - H(\cdot,{\tn'}) \right|   \right. \\
  \left. + \sup_{x \in \Xset} \| P_\tn(x,\cdot) - P_{\tn'}(x,\cdot) \|_\tv +
    \|\pi_\tn \,\rmd \lambda - \pi_{\tn'} \,\rmd \lambda\|_\tv \right) \eqsp.
\end{multline*}
Then, by \cite[Lemmas 4.6. and 4.7]{fort:jourdain:kuhn:lelievre:stoltz:2014},
there exists a constant $C$ such that for any $\tn, \tn' \in \Theta$
\begin{multline*}
  \sup_\Xset \left| P_{\tn}\hatH_{\tn} - P_{\tn'}\hatH_{\tn'} \right|  \\
  \leq C \left( |\tn - \tn'| + \sum_{i=1}^d \left|1 - \frac{\tn'(i)}{\tn(i)}
    \right| +\sum_{i=1}^d \left|1 - \frac{\tn(i)}{\tn'(i)} \right| \right)
  \eqsp.
\end{multline*}
Since $\sup_{\tn \in \Theta} \sup_{x \in \Xset} |H(x,\tn)| \leq 1$ and
$\P(\sup_n |\Lambda_{n+1}| \leq \gamma_1) =1$, (\ref{eq:theta:approxsto})
implies that $|\tn_{n+1} - \tn_n| \leq (1+\gamma_1) \gamma_{n+1}$ with
probability one.  By (\ref{eq:theta:update:multiplicatif}), for any $i
\in \{1, \cdots, d \}$
\begin{align}
  &\left|1 - \frac{\tn_n(i)}{\tn_{n+1}(i)} \right| \vee \left|1 -
  \frac{\tn_{n+1}(i)}{\tn_n(i)} \right|\notag\\
  & = \frac{|\tn_{n+1}(i)-\tn_{n}(i)|}{\tn_{n}(i) \wedge \tn_{n+1}(i)} \notag \\
  & = \frac{\gamma_{n+1}|\theta_n(I(X_{n+1}))-\un_{\Xset_i}(X_{n+1})|}{1+\gamma_{n+1}(\un_{\Xset_i}(X_{n+1})\wedge\theta_n(I(X_{n+1}))}\leq \gamma_{n+1} \quad \text{a.s.}\label{eq:majoevolrappthet}
\end{align}
This discussion evidences that there exists a constant $C$ such that
\[
\P\left( \forall k, \, \sup_{\ell \geq k} \left|\sum_{n=k}^\ell \gamma_{n+1}
    R_{n+1}^{(2)} \right| \leq C \, (1+\gamma_1) \sum_{n \geq k} \gamma_{n+1}^2 \right) =1
\eqsp.
\]
By (\ref{eq:stepsize:assumptions}), $\sup_{\ell \geq k}
\left|\sum_{n=k}^\ell \gamma_{n+1} R_{n+1}^{(2)} \right|$ tends to zero a.s. as
$k \to \infty$.

\textit{(ii)} The proof follows the same lines as the proof of \cite[Theorem
3.4]{fort:jourdain:kuhn:lelievre:stoltz:2014} and details are omitted. 
The only result which has to be 
adapted is \cite[Corollary 4.8]{fort:jourdain:kuhn:lelievre:stoltz:2014}. Combining \cite[Lemmas~4.6. and 4.7]{fort:jourdain:kuhn:lelievre:stoltz:2014} and \eqref{eq:majoevolrappthet}, we easily obtain the existence of a finite constant $C$
  such that almost surely, for any $n \geq 1$
  \begin{align*}
& \| \pi_{\tn_{n+1}} \, \rmd \lambda - \pi_{\tn_{n}} \, \rmd \lambda
   \|_{\mathrm{TV}} \leq C \gamma_{n+1} \eqsp, \\
     & \sup_{x \in \Xset}  \|P_{\tn_n}(x, \cdot) - P_{\tn_{n+1}}(x,\cdot)
     \|_{\mathrm{TV}} \leq C \gamma_{n+1}
      \eqsp.
   \end{align*}
\textit{(iii)} The proof is very similar to the proof of 
\cite[Theorem~3.5]{fort:jourdain:kuhn:lelievre:stoltz:2014} and is therefore omitted.

\subsection{Proof of Proposition~\ref{lem:bounds:BigThetaTilde}}

Throughout this proof, we denote by 
\[
\S_n = \sum_{i=1}^d \tu_n(i)
\]
the sum of the unnormalized weights. 
In view of~\eqref{eq:def:thetatilde}
\[
\S_{n+1}=\S_{n}+{\pas} \tn_n(I(X_{n+1})) \eqsp.
\]
As $\max_{1\leq i\leq d} \tn_n(i)\leq 1$, a direct induction on $n$ yields
$\S_n\leq \S_0+n \pas$. Since by~\eqref{eq:def:stepsize} 
\begin{equation}
  \gamma_{n+1}=\frac{\pas}{\S_n},\label{lienpassn}
\end{equation} 
the lower bound on $\gamma_{n+1}$ in Proposition~\ref{lem:bounds:BigThetaTilde} follows. To prove the deterministic upper-bound, we remark that
\begin{align}
  \S_{n+1}^2 &=\left(\S_n+\pas
    \frac{\tu_n(I(X_{n+1}))}{\S_n}\right)^2  \nonumber \\
  & \geq \S_n^2+2\pas \tu_n(I(X_{n+1})) \label{evolcnor2} \\
 & \geq \S_n^2+2\pas\min_{1\leq i\leq d}\tu_0(i),\nonumber
\end{align}
where we used that for each $i\in\{1,\ldots,d\}$, the sequence $(\tu_n(i))_{n \geq 0}$
is non-decreasing. By induction on $n$, this implies that $\S_n^2\geq
\S_0^2\left(1+2n\gamma_1\min_{1\leq i\leq d}\tn_0(i)\right)$ and the
deterministic upper-bound follows from \eqref{lienpassn}. 

To prove the stochastic upper-bound,  we suppose  A\ref{hyp:targetpi} and A\ref{hyp:kernel}. The deterministic upper-bound and Proposition \ref{prop:stabilite} ensure that \eqref{lgtk} holds. Let 
\[
\underline{\tu}_n=\min_{1\leq i\leq d}\tu_n(i).
\]
We have $\underline{\tu}_n= \S_n \, \underline{\theta}_n$ for
each $n\in\N$. Moreover the sequence $(\underline{\tu}_n)_{n\in\N}$
is non-decreasing and such that \begin{equation} \forall
  k\in\N^*,\;\underline{\tu}_{T_k}\geq\underline{\tu}_{T_{k-1}}\left(1+\gamma_{T_k}\right),\label{evolmintt}
\end{equation}
with equality when the smallest index of stratum with smallest weight $I_n$ is constant for $T_{k-1}\leq n\leq T_k$.
The inequality is due to the possibility that for some
$n\in\{T_{k-1}+1,T_{k-1}+2,\ldots,T_{k}-1\}$, $X_n\in\Xset_{I_{n-1}}$ and
$\exists i\in\{1,\ldots,I_{n-1}-1,I_{n-1}+1,\ldots,d\}$ such
$\tu_{n-1}(i)<\tu_{n-1}(I_{n-1})+\pas{\theta_{n-1}}(I_{n-1})$ so that
$I_{n}\neq I_{n-1}$. With~\eqref{lgtk} and the lower-bound in
Proposition~\ref{lem:bounds:BigThetaTilde} (recall that $\gamma_1 = \frac{\gamma}{S_0}$), one deduces that
\begin{align*} 
  \underline{\tu}_{T_k} 
  & \geq \underline{\tu}_{0}\prod_{j=1}^k (1+\gamma_{T_j})  \geq \underline{\tu}_{0}\prod_{j=1}^k \left(1 + \frac{\gamma}{S_0 +
    \gamma (T_j -1)}\right)\\
  & \geq \underline{\tu}_{0}\prod_{j=1}^k \left(1 + \frac{\gamma}{S_0 +
    \gamma C_T j}\right)\\
  & = \underline{\tu}_{0}\prod_{j=1}^k\frac{\S_0+\pas+C_T\pas j}{\S_0+C_T\pas j} \\
  &=\underline{\tu}_{0}\frac{\Gamma\left(\frac{\S_0+\pas}{C_T\pas}+k+1\right)}{\Gamma\left(\frac{\S_0}{C_T\pas}+k+1\right)}
  \frac{\Gamma\left(\frac{\S_0}{C_T\pas}+1\right)}{\Gamma\left(\frac{\S_0+\pas}{C_T\pas}+1\right)} \\
  & \sim
  \underline{\tu}_{0}k^{1/C_T}\frac{\Gamma\left(\frac{\S_0}{C_T\pas}+1\right)}{\Gamma\left(\frac{\S_0+\pas}{C_T\pas}+1\right)}\mbox{
    as }k\to+\infty \eqsp.
\end{align*}
Hence there is a positive random variable $C$ such that, for all $k\in \N$,
$\underline{\tu}_{T_k}\geq C(k+1)^{1/C_T}$. Since by \eqref{lgtk},
$T_{\lfloor n/C_T\rfloor}\leq n$ and the sequence $(\underline{\tu}_{n})_{n \geq 0}$ is
non-decreasing, it follows that
\[
\forall n\in\N, \quad \underline{\tu}_{n}\geq\underline{\tu}_{T_{\lfloor n/C_T\rfloor}} \geq C\left(\frac{n}{C_T}\right)^{1/C_T} \quad \mathrm{a.s.}
\]
Since by \eqref{evolcnor2}, $\S_n^2\geq
\S_0^2+2\pas\sum_{j=0}^{n-1}\underline{\tu}_{j}$, this implies that a.s., for
any $n \geq 1$,
\begin{align*}
  \S^2_n & \geq \S_0^2 + 2 \gamma C \sum_{j=0}^{n-1} \left(\frac{j}{C_T}\right)^{1/C_T}\\
  & \geq \S_0^2+2\pas CC_T^{-1/C_T}\int_0^{n-1}x^{1/C_T} \rmd x \\
  & =\S_0^2+2\pas CC_T^{-1/C_T}\frac{C_T}{1+C_T}(n-1)^{\frac{1+C_T}{C_T}}
  \eqsp.
\end{align*}
With \eqref{lienpassn}, one deduces that
\[
\P\left(\sup_{n\in\N}n^{\frac{1+1/C_T}{2}}\gamma_{n+1}<\infty\right)=1.
\]

\subsection{Proof of Proposition~\ref{propshusa}}

Let us consider the SHUS$^\alpha$ algorithm for fixed $\alpha\in
(1/2,1)$. For notational simplicity, we omit the dependence of $\gamma(\alpha)$ on $\alpha$. The proof of Proposition~\ref{propshusa} relies on the next lemma.

\begin{lemma}\label{lem:shusa} 
Under Assumptions A\ref{hyp:targetpi} and A\ref{hyp:kernel}, the
sequence $(S_n)_{n \geq 1} =\left(\sum_{i=1}^d\tu_n(i)\right)_{n\geq 0}$ is increasing and bounded from below by a deterministic sequence $(\underline{s}_n)_{n\geq 0}$ going to $+\infty$ as $n\to\infty$, and satisfies 
\begin{align*}
   \P\bigg(0&<\inf_{n\geq 0}(n+1)^{\alpha-1}\ln(1+S_n)\\&\leq \sup_{n\geq 0}(n+1)^{\alpha-1}\ln(1+S_n)\leq \bar{c}\bigg)=1,
\end{align*} 
where
\[
\begin{aligned}
\bar{c}& = \ln(1+S_0) \\ 
& \times \left(1\vee\left(\left(1+\frac{\gamma}{(\ln(1+S_0))^{\frac{1}{1-\alpha}}}\right)^{\frac{1}{1-\alpha}}\!\!-1\right)^{1-\alpha}\right).
\end{aligned}
\] 
\end{lemma}

\begin{proof}[of Lemma~\ref{lem:shusa}]
The sequence $(S_n)_{n\geq 0}$ increases according to 
\begin{equation}
  \begin{aligned}
   S_{n+1} &= S_n+\frac{\gamma \tu_n(I(X_{n+1}))}{\left(\ln(1+S_n)\right)^{\frac{\alpha}{1-\alpha}}} \\
   & = S_n\left(1+\frac{\gamma \tn_n(I(X_{n+1}))}{\left(\ln(1+S_n)\right)^{\frac{\alpha}{1-\alpha}}}\right).\label{evolsn}
\end{aligned}
\end{equation} 
Since for $i\in\{1,\hdots,d\}$, the sequence $(\tu_n(i))_{n\geq 0}$ is non-decreasing, one deduces that, for all $n\geq 0$, $S_{n+1}\geq g(S_n)$, where 
\[
g(x)=x+\frac{\gamma\min_{1\leq i\leq d}\tu_0(i)}{(\ln(1+x))^{\frac{\alpha}{1-\alpha}}}
\]
for $x >0$. Let $(\underline{s}_n)_{n\geq 0}$ be defined inductively by $\underline{s}_0=S_0$ and $\underline{s}_{n+1}=g(\underline{s}_{n})$ for all $n\geq 0$. This sequence is increasing and goes to $\infty$ when $n\to\infty$ as $x\mapsto \gamma\left[ \min_{1\leq i\leq d}\tu_0(i)\right](\ln(1+x))^{-\frac{\alpha}{1-\alpha}}$ is locally bounded away from $0$ on $(0,+\infty)$. Moreover $g$ is a convex function which is decreasing on $(0,x_g)$ and increasing on $(x_g,+\infty)$ for some $x_g\in (0,+\infty)$. The minimum of the function~$g$ is therefore attained at~$x_g$. Since 
\[
S_1\geq \underline{s}_1=g(S_0)\geq g(x_g)=x_g+\frac{\gamma\min_{1\leq i\leq d}\tu_0(i)}{(\ln(1+x_g))^{\frac{\alpha}{1-\alpha}}}>x_g,
\]
it follows that $\min(S_n,\underline{s}_n)> x_g$ for any $n\geq 1$, and one easily checks by induction on $n$ that $S_n\geq\underline{s}_n$ for all $n\geq 1$. Unfortunately, this lower-bound is not sharp enough to bound $((n+1)^{\alpha-1}\ln(1+S_n))_{n\geq 0}$ from below by a positive constant (this will proved later on).

To prove that 
\begin{equation}
   \forall n\geq 0,\;\ln(1+S_n)\leq \bar{c}(n+1)^{1-\alpha},\label{minolnsn}
\end{equation} 
we remark that for all $n\in \N$, using $S_{n+1}\geq S_n$ for the first inequality, and $S_{n+1}\leq S_n\left(1+\frac{\gamma}{(\ln(1+S_n))^{\frac{\alpha}{1-\alpha}}}\right)$ for the second one,
\begin{align}
   \ln(1+S_{n+1})&\leq\ln(1+S_n)+\ln\left(\frac{S_{n+1}}{S_n}\right)\notag\\&\leq \ln(1+S_n)+\frac{\gamma}{(\ln(1+S_n))^{\frac{\alpha}{1-\alpha}}}\label{majodifl1+s}.
\end{align}
Therefore, denoting for simplicity $l_n=\left(\ln(1+S_n)\right)^{\frac{1}{1-\alpha}}$ and using the monotonicity of the sequence $(l_n)_{n\geq 0}$ and the convexity of $x\mapsto (1+x)^{\frac{1}{1-\alpha}}$ on $\R_+$ for the second inequality, we get
\begin{align*}
  l_{n+1}&\leq l_n\left(1+\frac{\gamma}{l_n}\right)^{\frac{1}{1-\alpha}}\\
&\leq l_n\left(1+\frac{l_0}{l_n}\left(\left(1+\frac{\gamma}{l_0}\right)^{\frac{1}{1-\alpha}}-1\right)\right).
\end{align*}
This implies by induction on $n$ that
\begin{align*}
   l_n&\leq l_0+nl_0\left(\left(1+\frac{\gamma}{l_0}\right)^{\frac{1}{1-\alpha}}-1\right)\\
&\leq l_0\left(1\vee \left(\left(1+\frac{\gamma}{l_0}\right)^{\frac{1}{1-\alpha}}-1\right)\right)(n+1).
\end{align*}
Raising this inequality to the power $1-\alpha$ leads to \eqref{minolnsn}, which in turn implies the lower bound
\begin{equation}
   \forall n\geq 1,\;\gamma_n\geq \gamma{\bar{c}}^{\frac{-\alpha}{1-\alpha}}n^{-\alpha}\label{minnalpas}.\end{equation}

To bound $(n+1)^{\alpha-1}\ln(1+S_n)$ from below, we are going to adapt the proof of Proposition \ref{lem:bounds:BigThetaTilde}. Since $(\gamma_n)_{n\geq 1}$ is bounded from above by the deterministic sequence $(\gamma(\ln(1+\underline{s}_n))^{-\frac{\alpha}{1-\alpha}})_{n\geq 1}$ which goes to $0$ as $n\to\infty$, Proposition~\ref{prop:stabilite} shows that the sequence $(T_k)_{k\geq 0}$ defined inductively by $T_0=0$ and~\eqref{eq:Tk} satisfies \eqref{lgtk}.
Moreover, \eqref{evolmintt} still holds so that,  using the concavity of the logarithm and the monotonicity of the sequence $(\gamma_n)_{n\geq 1}$ for the second inequality then setting $c=\frac{\gamma\ln(1+\gamma_1)}{C_T^\alpha\gamma_1{\bar{c}}^{\alpha/(1-\alpha)}}$ and using \eqref{minnalpas} and \eqref{lgtk} for the third, one has
\begin{align*}
   \ln \underline{\tu}_{T_k} 
  & \geq \ln \underline{\tu}_{0}+\sum_{j=1}^k \ln (1+\gamma_{T_j}) \\&\geq \ln \underline{\tu}_{0}
+\frac{\ln(1+\gamma_1)}{\gamma_1}\sum_{j=1}^k\gamma_{T_j}\\
&\geq\ln \underline{\tu}_{0}+c\sum_{j=1}^kj^{-\alpha}\\&\geq \ln \underline{\tu}_{0}+\frac{c}{1-\alpha}\left((k+1)^{1-\alpha}-1\right). \end{align*}
With \eqref{lgtk}, one deduces that a.s., for all $n\in \N$,
\begin{align*}
   \underline{\tu}_n&\geq \underline{\tu}_{T_{\lfloor n/C_T\rfloor}} \\&\geq \underline{\tu}_{0}\exp\left(\frac{c}{1-\alpha}\left(\frac{(n+1)^{1-\alpha}}{C_T^{1-\alpha}}-\frac{1}{C_T^{1-\alpha}}-1\right)\right).
\end{align*}
Inserting this lower-bound together with \eqref{minolnsn} into
\eqref{evolsn} and setting $c_0=\frac{\gamma\underline{\tu}_{0}}{{\bar{c}}^{\frac{\alpha}{1-\alpha}}}\exp
  \left(-\frac{c(1+C_T^{1-\alpha})}{(1-\alpha)C_T^{1-\alpha}} \right)$ and $c_1=\frac{c}{(1-\alpha)C_T^{1-\alpha}}$ one gets
\[
\forall n\in\N,\;S_{n+1}\geq S_n+\frac{c_0}{(n+1)^\alpha}\mathrm{e}^{c_1(n+1)^{1-\alpha}}.
\]
Since $x\mapsto x^{-\alpha}\mathrm{e}^{c_1x^{1-\alpha}}$ is increasing for $x\geq \left(\frac{\alpha}{c_1(1-\alpha)}\right)^{\frac{1}{1-\alpha}}$, one deduces that for all $n\geq n_1:=\left\lceil\left(\frac{\alpha}{c_1(1-\alpha)}\right)^{\frac{1}{1-\alpha}}\right\rceil$,
\begin{align*}
   S_n&\geq
   S_{n_1}+\int_{n_1}^nc_0x^{-\alpha}\mathrm{e}^{c_1x^{1-\alpha}} \rmd x\\&=S_{n_1}+\frac{c_0}{c_1(1-\alpha)}\left(\mathrm{e}^{c_1n^{1-\alpha}}-\mathrm{e}^{c_1 n_1^{1-\alpha}}\right)
\end{align*}
so that $\limsup_{n\to\infty}(n+1)^{\alpha-1}\ln(1+S_n)\geq c_1>0$. This concludes the proof since $(n+1)^{\alpha-1}\ln(1+S_n)>0$ for all $n\geq 0$.
\end{proof}

\begin{proof}[of Proposition~\ref{propshusa}]
By Lemma~\ref{lem:shusa}, the sequence $(\gamma_n)_{n\geq 1}$ defined by \eqref{eq:gammanSHUSalpha} is increasing, and bounded from above by the deterministic sequence \linebreak $\left(\gamma(\ln(1+\underline{s}_{n-1}))^{-\frac{\alpha}{1-\alpha}}\right)_{n\geq 1}$ which goes to $0$ as $n\to\infty$. Moreover,
\[
\P\left(\gamma{\bar{c}}^{-\frac{\alpha}{1-\alpha}}\leq \inf_{n\geq 1}n^\alpha\gamma_n\leq\sup_{n\geq 1}n^\alpha\gamma_n <+\infty\right)=1.
\]
Since $\sum_{n \geq 1}\frac{1}{n^\alpha}=+\infty$ and $\sum_{n \geq
  1}\frac{1}{n^{2\alpha}}<+\infty$, one deduces that \eqref{eq:stepsize:assumptions} holds. The convergence is therefore a consequence of Theorem~\ref{theo:cvggal}.

To prove~\eqref{eq:limit_gamma_n_shus_alpha}, we remark that \eqref{evolsn} implies that
\[
\ln(1+S_{n+1})=\ln(1+S_n)+\frac{\gamma\tn_n(I(X_{n+1}))}{(\ln(1+S_n))^{\frac{\alpha}{1-\alpha}}}+R^{(1)}_{n+1},
\]
where, setting $h(x)=\ln(1+x)-x$,
\begin{align*}
   R^{(1)}_{n+1}=&h\left(\frac{\gamma S_n\tn_n(I(X_{n+1}))}{(1+S_n)(\ln(1+S_n))^{\frac{\alpha}{1-\alpha}}}\right)\\&-\frac{\gamma \tn_n(I(X_{n+1}))}{(1+S_n)(\ln(1+S_n))^{\frac{\alpha}{1-\alpha}}}.
\end{align*}
From now on, $C$ denotes a positive random variable which may change from line to line. Lemma \ref{lem:shusa} implies that
\[
\forall n\geq 0,\;\ln(1+S_n)\geq C(n+1)^{1-\alpha}.
\]
Since $0\geq h(x)\geq-x^2/2$ for all $x > 0$ and
\[
\sup_{x>0}\frac{(\ln(1+x))^{\frac{\alpha}{1-\alpha}}}{1+x}<+\infty, 
\]
it follows that 
\[
\forall n\geq 0,\;\left|R^{(1)}_{n+1}\right|\leq c(\ln(1+S_n))^{-\frac{2\alpha}{1-\alpha}}
\]
for some deterministic constant $c\in(0,+\infty)$.
Writing 
\begin{align*}
   (&\ln(1+S_{n+1}))^{\frac{1}{1-\alpha}}=(\ln(1+S_{n}))^{\frac{1}{1-\alpha}}\\&\times\left(1+\frac{\gamma\tn_n(I(X_{n+1}))}{(\ln(1+S_n))^{\frac{1}{1-\alpha}}}+\frac{R^{(1)}_{n+1}}{\ln(1+S_n)}\right)^{\frac{1}{1-\alpha}}
\end{align*}
and remarking that $x\mapsto \frac{1}{x^2}\left|(1+x)^{\frac{1}{1-\alpha}}-1-\frac{x}{1-\alpha}\right|$ is locally bounded on $\R_+$, one deduces that
\[
\begin{aligned}
& (\ln(1+S_{n+1}))^{\frac{1}{1-\alpha}} \\
& =(\ln(1+S_{n}))^{\frac{1}{1-\alpha}}+\frac{\gamma\tn_n(I(X_{n+1}))}{1-\alpha}+R^{(2)}_{n+1},
\end{aligned}
\]
where, for all $n\geq 0$, $\left|R^{(2)}_{n+1}\right|\leq c'(\ln(1+S_n))^{-\frac{\alpha}{1-\alpha}}$ for some deterministic constant $c'\in(0,+\infty)$ depending on $S_0$. Lemma \ref{lem:shusa} ensures the existence of a positive random variable $C$ such that $\forall n\geq 0$, $|R^{(2)}_{n+1}|\leq C(n+1)^{-\alpha}$. One has
\begin{align*}
   \frac{1}{n}\Big(\ln(1+S_{n})\Big)^{\frac{1}{1-\alpha}}=&\frac{1}{n}\Big(\ln(1+S_{0})\Big)^{\frac{1}{1-\alpha}}+\frac{1}{n}\sum_{k=1}^n R^{(2)}_k\\&+\frac{\gamma}{(1-\alpha)n}\sum_{k=1}^n\tn_{k-1}(I(X_{k})).
\end{align*}
The first term in the right-hand side converges a.s. to $0$ as $n\to\infty$. So does the second since 
\[
\frac{1}{n}\sum_{k=1}^nk^{-\alpha}\leq \frac{1}{n}\int_0^nx^{-\alpha}\rmd x=\frac{n^{-\alpha}}{1-\alpha}.
\]
The choice $f\equiv 1$ in \eqref{theo:ergoLLN} ensures that the third term converges a.s. to $\frac{\gamma}{d(1-\alpha)}$. One concludes that $n^\alpha \gamma_{n+1}=\gamma\left(\frac{1}{n}(\ln(1+S_{n}))^{\frac{1}{1-\alpha}}\right)^{-\alpha}$ converges a.s. to $\gamma\left(\frac{\gamma}{d(1-\alpha)}\right)^{-\alpha}$.

\end{proof}

\end{document}